\documentclass[10pt]{amsart}
\makeatletter
\def\blfootnote{\gdef\@thefnmark{}\@footnotetext}
\makeatother
\usepackage{epsfig}
\usepackage{graphics}

\usepackage{amscd,amsthm,amsfonts}

\usepackage[colorlinks=true,
                    linkcolor=blue,
                    urlcolor=blue,
                    citecolor=blue,
                    anchorcolor=blue]{hyperref}
\usepackage{amssymb}
\usepackage{footnote}
\usepackage{multicol}
\usepackage[pagewise]{lineno}
\newtheorem{theorem}{Theorem}[section]
\newtheorem{lemma}[theorem]{Lemma}
\newtheorem{proposition}[theorem]{Proposition}
\newtheorem{corollary}[theorem]{Corollary}

\theoremstyle{definition}

\newtheorem{remark}[theorem]{Remark}

\theoremstyle{remark}

\numberwithin{equation}{section}

\newcommand\addtag{\refstepcounter{equation}\tag{\theequation}}

 \def\C{{\mathbb{C}}}
  \def\P{{\mathbb{P}}}

\def\Mod{{\rm Mod}}

\def\U{{\rm U}}
\def\Diff{{\rm Diff}}

\begin{document}
 \blfootnote{\textup{2000} \textit{Mathematics Subject Classification}:
57R55, 57R17}

\newenvironment{prooff}{\medskip \par \noindent {\it Proof}\ }{\hfill
$\square$ \medskip \par}
    \def\sqr#1#2{{\vcenter{\hrule height.#2pt
        \hbox{\vrule width.#2pt height#1pt \kern#1pt
            \vrule width.#2pt}\hrule height.#2pt}}}
    \def\square{\mathchoice\sqr67\sqr67\sqr{2.1}6\sqr{1.5}6}
\def\pf#1{\medskip \par \noindent {\it #1.}\ }
\def\endpf{\hfill $\square$ \medskip \par}
\def\demo#1{\medskip \par \noindent {\it #1.}\ }
\def\enddemo{\medskip \par}
\def\qed{~\hfill$\square$}

 \title[Lefschetz Fibrations and Exotic $4$-Manifolds]
{Genus-$3$ Lefschetz Fibrations and Exotic $4$-Manifolds with $b_{2}^{+}=3$}
\author[T{\"{u}}l\.{i}n Altun{\"{o}}z]{T{\"{u}}l\.{i}n Altun{\"{o}}z}
 \address{Department of Mathematics, Middle East Technical University,
  Ankara, Turkey}
  \email{atulin@metu.edu.tr}
\date{\today}
\maketitle
 \begin{abstract}
We explicitly construct a genus-$3$ Lefschetz fibration over $\mathbb{S}^{2}$ whose 
total space is $\mathbb{T}^{2}\times \mathbb{S}^{2}\# 6\overline{\C\P^{2}}$
using the monodromy of Matsumoto's genus-$2$ Lefschetz fibration. 
We then construct more genus-$3$ Lefschetz fibrations whose total spaces are exotic minimal symplectic $4$-manifolds  
$3 \C\P^{2} \#  q\overline{\C\P^{2}}$ for $q=13,\ldots,19$. We also generalize our construction to get genus-$3k$ Lefschetz fibration structure on the $4$-manifold $\Sigma_{k}\times \mathbb{S}^{2}\# 6\overline{\C\P^{2}}$ using the generalized Matsumoto's genus-$2k$ Lefschetz fibration. From this generalized version, we derive further exotic 
$4$-manifolds via Luttinger surgery and twisted fiber sum.
\end{abstract}

%\tableofcontents
  \setcounter{secnumdepth}{2}
 \setcounter{section}{0}

\section{Introduction}\label{S1}

There is a close relationship between objects in $4$-dimensional topology and algebra
 by virtue of the pioneering works of Donaldson and Gompf. By remarkable work of
 Donaldson, it was shown that every closed symplectic $4$-manifold has a structure
 of a Lefschetz pencil which yields a Lefschetz
 fibration after blowing up at its  base points ~\cite{d2}. Conversely, Gompf~\cite{gs} showed that the total space of a
 genus-$g$ Lefschetz fibration admits a symplectic structure if $g\geq 2$. This
 relation provides a way to understand any
 symplectic $4$-manifold via a positive factorization of its monodromy, if it exists.
 Given a genus-$g$ Lefschetz fibration $f:X \to \mathbb{S}^{2}$, one can 
 associate to it the factorization of the identity word $W=1$ in the mapping class 
 group of the closed connected orientable genus-$g$ surface. Conversely, given such a factorization, one can construct a 
 genus-$g$ Lefschetz fibration over $\mathbb{S}^{2}$.\par
Proving the existence of minimal symplectic structures on $4$-manifolds and 
constructing such manifolds in the homeomorphism classes of simply connected 
$4$-manifolds with very small topology have been an interesting topic that has used several 
construction techniques such as rational blowdowns, knot surgery, 
fiber sums and Luttinger surgeries. (e.g. ~\cite{a1,a3,bk,d1,f1,f2,f3,fmor,g1,ko,p1,p2,sz}.) 
Recently, some authors have applied some relations in the 
mapping class group such as lantern relation or Luttinger surgery to construct 
Lefschetz fibrations with $b_{2}^{+}=1$ or $3$ ~\cite{a5,a6,a4,a7,b1,bh,bhm,bki,e2,e4,e3}. For instance,  in ~\cite{am,a4,bki},  
genus-$2$ Lefschetz fibrations are studied and exotic genus-$2$ 
Lefschetz fibrations with $b_{2}^{+}\leq3$ are obtained via several constructions, especially using their monodromies. Also, Akhmedov and Monden constructed 
some higher genus fibrations via lantern and daisy substitutions~\cite{am1}. 
We would like to specify that the aim of this paper is not only to construct 
exotic smooth structures on very small $4$-manifolds with $b_{2}^{+}=3$, 
but also to use Lefschetz fibration structures on various smooth $4$-manifolds with small numbers of singular fibers. The novel aspects of this paper is the following. We apply the breeding technique to produce a small genus-$3$ fibration, and follow the Baykur-Korkmaz scheme~\cite{bki} to generate fibered exotic $4$-manifolds through small Lefschetz fibrations, and apply 
Luttinger surgeries to get smaller $4$-manifolds. We also apply twisted fiber sum operations and lantern substitutions which correspond to the symplectic rational blowdown surgery along $-4$ sphere~\cite{e2} to get fibered genus-$3$ exotic $4$-manifolds.

In the present paper, we first construct a relation $W=1$ in the mapping class group of
the closed connected orientable genus-$3$ surface, denoted by $\Mod_{3}$, using Matsumoto's 
well known relation~\cite{mat}. We use the construction  technique given by Baykur and 
Korkmaz in~\cite{bk1} to get the relation $W=1$ in $\Mod_{3}$ (see~\cite{b1} for more examples of this technique). We then 
obtain a genus-$3$ Lefschetz fibration $f:X\to \mathbb{S}^2$
with the monodromy factorization $W$ and prove that the $4$- manifold $X$ is diffeomorphic to $\mathbb{T}^{2}\times \mathbb{S}^{2}\# 6\overline{\C\P^{2}}
$. We apply lantern 
substitutions to the twisted fiber sums of the genus-$3$ Lefschetz fibration $f:X\to \mathbb{S}^2$ to obtain minimal genus-$3$ Lefschetz fibrations whose total spaces are homeomorphic but not diffeomorphic to 
$3 \C\P^{2} \#  q\overline{\C\P^{2}}$ for $q=13,14,15$. We also construct simply connected genus-$3$ Lefschetz fibrations applying lantern substitutions to the twisted 
fiber sums of the Lefschetz fibration $f:X\to \mathbb{S}^2$
and the generalized Matsumoto's fibration for $g=3$  obtained by ~\cite{k1,Cadavid,dmp}. Hence, we obtain exotic minimal symplectic 
$4$-manifolds admitting genus-$3$ Lefschetz fibration structures in the homeomorphism classes of 
$3\C\P^{2} \#  q\overline{\C\P^{2}}$ for $q=16,\ldots,19$. Therefore, the first main result of this paper is as follows.
 \begin{theorem}\label{tt1}
 There exist genus-$3$ Lefschetz fibrations whose total spaces are minimal symplectic $4$-manifolds homeomorphic but not diffeomorphic to 
  $3\C\P^{2} \#  q\overline{\C\P^{2}}$ for $q=13,\ldots,19$.
 \end{theorem}
We would like to remark that the idea of constructing small exotic $4$-manifolds admitting a genus-$3$ Lefschetz pencil/fibration structure via positive factorizations is already present in the literature: as shown by Baykur~\cite{b1}, there exist genus-$3$ Lefschetz pencils on exotic copies of  $\C\P^{2} \#  p\overline{\C\P^{2}}$ for $p=7,8,9$. Genus-$3$ Lefschetz fibrations on exotic copies of $3\C\P^{2} \#  q\overline{\C\P^{2}}$ for $q=16,\ldots,19$ have been constructed by Baykur, Hayano and Monden~\cite{bhm}. Also, Baykur and Hayano~\cite{bh} gave examples of exotic genus-$3$ Lefschetz pencils.

Following Theorem~\ref{tt1}, we generalize our relation $W=1$ in $\Mod_{3}$ to the relation $W_k=1$ in $\Mod_{3k}$. We then obtain the $4$-manifold $\Sigma_{k}\times \mathbb{S}^{2}\# 6\overline{\C\P^{2}}$ that admits a genus-$3k$ Lefschetz fibration over $\mathbb{S}^2$ with monodromy $W_k=1$. Using this Lefschetz fibration structure, we produce exotic copies of
$(4k-1) \C\P^{2} \#  (4k+5)\overline{\C\P^{2}}$
for any positive integer $k$ via Luttinger surgery. Finally, we
construct minimal exotic copies of
$(4k^{2}-2k+1) \C\P^{2} \#  (4k^{2}+4k+7)\overline{\C\P^{2}}$
admitting the genus-$3k$ Lefschetz fibration structure for any positive integer $k$ via the twisted fiber sum. Hence, the other results we obtained in this paper are as follows.
 \begin{theorem}\label{tt2}
 There exist infinitely many smooth $4$-manifolds homeomorphic but not diffeomorphic to $(4k-1) \C\P^{2} \#  (4k+5)\overline{\C\P^{2}}$ for any positive integer $k$.
 \end{theorem}
 Note that infinite families of exotic $4$-manifolds in all homeomorphism classes given in Theorem~\ref{tt1} and~\ref{tt2} have been constructed in~\cite{abbkp}. There are even earlier examples in most of homeomorphism classes of $4$-manifolds in Theorem~\ref{tt2}, constructed in the pioneering works of Park~\cite{jp1,jp2}.
 
 \begin{theorem}
 There exist genus-$3k$ Lefschetz fibrations, whose
  total spaces are minimal symplectic $4$-manifolds homeomorphic but not diffeomorphic to $(4k^{2}-2k+1) \C\P^{2} \#  (4k^{2}+4k+7)\overline{\C\P^{2}}$ for any positive integer $k$.
 \end{theorem}
The present paper is organized as follows. In Section~\ref{S2}, we give some preliminary
information and results. In Section ~\ref{S3}, we explicitly construct a positive factorization
$W$ for a genus-$3$ Lefschetz fibration over $\mathbb{S}^{2}$. We also derive new
simply-connected genus-$3$ Lefschetz fibrations by constructing their monodromies. In Section ~\ref{S4}, we construct minimal symplectic $4$-manifolds admitting a 
genus-$3$ Lefschetz fibration structure which are exotic copies of
$3\C\P^{2} \#  q\overline{\C\P^{2}}$ for $q=13,\ldots,19$ using
genus-$3$ Lefschetz fibrations obtained in Section $3$ and generalized Matsumoto's
fibration for $g=3$. In the last section, we generalize the construction of genus-$3$
Lefschetz fibration obtained in Section~\ref{S3} to obtain genus-$3k$ Lefschetz
fibrations for any positive integer $k$. Finally, using these genus-$3k$ Lefschetz fibration structures, we derive
some fibered and non-fibered examples of exotic structures.
\medskip

\noindent
{\bf Acknowledgements.}
I would like to thank my advisor Mustafa Korkmaz for his support and inspiring discussions
in this work. Thanks are due to Anar Akhmedov for helpful conversations
and suggestions. This paper was written while I was visiting School of Mathematics,
University of Minnesota. I would like to thank School of Mathematics, U of M for
their hospitality and wonderful research environment. This paper is a part of the author's Ph.D. thesis ~\cite{a}  at Middle
East Technical University. The author was partially supported by the Scientific and Technologic Research Council of Turkey (TUBITAK).
%%%%%%%%%%%%%%%%%%%%%%%%%%%%%%%%%%%%%%%%%%%%%%%%

\section{ preliminaries}\label{S2}

In this section, we state some preliminary definitions and recall some useful facts concerning the
mapping class groups, relations between Dehn twists and  Lefschetz fibrations ~\cite{fm,gs,mar}.
We then give the Matsumoto's relation~\cite{mat} and its generalized
version~\cite{k1,Cadavid,dmp}. Also, we give a method to compute the signatures of Lefschetz fibrations introduced by Endo and Nagami~\cite{e11}.

\subsection{Mapping Class Groups}
\  Let $\Sigma_{g}^{n}$ denote a compact connected oriented smooth surface of genus $g$
with $n\geq 0$ boundary components. Let $\Diff^{+}(\Sigma_{g}^{n})$ denote the group
of all orientation-preserving self-diffeomorphisms of $\Sigma_{g}^{n}$ fixing all points on the
boundary and let $\Diff^{+}_{0}(\Sigma_{g}^{n})$  denote the subgroup of
$\Diff^{+}(\Sigma_{g}^{n})$  consisting of orientation-preserving self-diffeomorphisms of
$\Sigma_{g}^{n}$  isotopic to the identity. {\textit{The mapping class group}} of
$\Sigma_{g}^{n}$, denoted by $\Mod_{g}^{n}$, is defined to be the group of the isotopy
classes of self-diffeomorphisms of $\Sigma_{g}^{n}$, where diffeomorphisms and isotopies
fix all points on the boundary, i.e.,
\[
\Mod_{g}^{n}=\Diff^{+}(\Sigma_{g}^{n})/ \Diff^{+}_{0}(\Sigma_{g}^{n}).
\]
We will denote $\Mod_{g}^{n}$ and $\Sigma_{g}^{n}$ by $\Mod_g$ and $\Sigma_{g}$,
respectively, when $n=0$.\par

For a simple closed curve $a$ on a surface $\Sigma_{g}^{n}$, a {\textit{right (or  positive)}}
 {{\textit{Dehn twist}}} about $a$ is the diffeomorphism $t_{a}$ obtained
by cutting $\Sigma_{g}^{n}$ along $a$ and gluing back after rotating one of the copies of $a$ by $360$ degrees to the right.\par

Throughout this paper, for any two mapping classes $f$ and $h$, the multiplication $fh$
means that $h$ is applied first and then $f$.

\begin{lemma}\label{le}\cite{fm} For any $f \in \Mod_{g}^{n}$ and two simple closed curves $a$ and $b$ on $\Sigma_{g}^{n}$, we have
 \begin{enumerate}
    \item $ft_{a}f^{-1}=t_{f(a)}$ ({{\textit{Conjugation}}}),
    \item if $a$ and $b$ are disjoint, then $t_a$ and $t_b$ commute ({{\textit{Commutativity}}}),
    \item if $a$ intersects $b$ transversely at one point, then  $t_at_bt_a=t_bt_at_b$ ({{\textit{Braid relation}}}).
  \end{enumerate}
\end{lemma}
\begin{figure}[ht]
\begin{center}
\scalebox{0.3}{\includegraphics{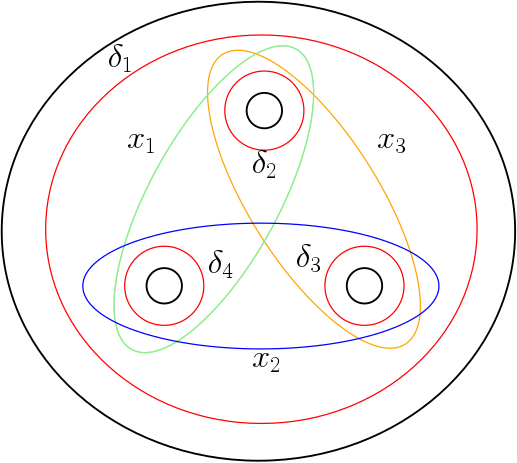}}
\caption{The lantern configuration on a sphere with four holes.}
\label{LR}
\end{center}
\end{figure}
\subsection{Lantern relation}
The lantern relation was discovered by Dehn and rediscovered by Johnson. 

 \begin{lemma}\cite{fm,j}
 \ Let $\delta_{1},\delta_{2}, \delta_{3}$ and $\delta_{4}$ be the boundary
curves of $\Sigma_{0}^{4}$ and $x_{1},x_{2}$ and $x_{3}$ be the simple closed
curves as shown in Figure~\ref{LR}. Suppose that $\Sigma_{0}^4$ is embedded
in $\Sigma_{g}^n$.
 The following relation, called the lantern relation, holds in $\Mod_{g}^{n}$:
 \[
 t_{\delta_{1}}t_{\delta_{2}}t_{\delta_{3}}t_{\delta_{4}}=t_{x_{1}}t_{x_{2}}t_{x_{3}}
 \] 
\end{lemma}

 \subsection{Lefschetz fibratons and monodromy representations}
 We start with a review of basic definitions and properties of Lefschetz fibrations. Throughout the paper, we denote the $2$-sphere by $\mathbb{S}^{2}$.

Let $M$ be a closed connected oriented smooth $4$-manifold. 
 A smooth surjective map $f:M\to \mathbb{S}^{2}$ is a {\textit{Lefschetz fibration}} 
 of genus $g$ if it has finitely many critical points and around each critical point, it can be written as 
 $f(z_{1},z_{2})=z_{1}^{2}+z_{2}^{2}$ with respect to some local complex coordinates agreeing with the orientations of $M $ and $\mathbb{S}^2$. The genus of a regular fiber $F$ is called 
\textit{the genus of the fibration}. We assume that all the critical points lie in the distinct fibers, called 
\textit{singular fibers}, which can be achieved after a small perturbation. Each singular fiber is 
obtained by shrinking a simple closed curve, called \textit{vanishing cycle}, to a point in the 
regular fiber. If the vanishing cycle is nonseparating (resp. separating), then the singular 
fiber is called \textit{irreducible} (resp. \textit{reducible}). In this work, we also assume 
that all Lefschetz fibrations are nontrivial, i.e. there exists at least one singular fiber and 
fibrations are relatively minimal, i.e. no fiber contains a $(-1)$-sphere, otherwise one can 
blow it down without changing the rest of the fibration.
 
Lefschetz fibrations can be described combinatorially via their monodromies. The monodromy
 of a Lefschetz fibration $f:M\to \mathbb{S}^{2}$ is given by a positive factorization 
$t_{\alpha_{1}}t_{\alpha_{2}}\cdots t_{\alpha_{n}}=1$ 
in $\Mod_g$, where $\alpha_{i}$'s are the vanishing cycles. Conversely,
for a given positive factorization $t_{a_{1}}t_{a_{2}}\cdots t_{a_{k}}=1$ in $\Mod_g$,
one can construct a genus-$g$ Lefschetz fibration over $\mathbb{S}^{2}$ by attaching
$2$-handles along vanishing cycles $a_{i}$ in a $\Sigma_{g}$ fiber in $\Sigma_{g}\times D^{2}$
with $-1$ framing with respect to the product framing, and then by closing it 
up via a fiber-preserving map between boundaries
to get a genus-$g$ fibration over $\mathbb{S}^{2}$. Such a fibration is uniquely determined up to
isomorphisms, which are orientation-preserving self-diffeomorphisms of the total spaces and
$\mathbb{S}^{2}$ making the fibrations commute.  The relation $t_{a_{1}}t_{a_{2}}\cdots t_{a_{k}}=1$
in $\Mod_g$ is uniquely determined up to \textit{Hurwitz moves} (exchanging subwords
$t_{a_{i}}t_{a_{i+1}}=t_{a_{i+1}}t_{t_{a_{i+1}^{-1}}(a_{i})}$ and \textit{global conjugations}
(changing each $t_{a_{i}}$ with $t_{\varphi(a_{i})}$ for some $\varphi \in \Mod_g$) if $g\geq 2$. 

A map $\sigma : \mathbb{S}^{2} \to M$ is called a \textit{section} of a Lefschetz fibration
$f:M\to \mathbb{S}^{2}$ if $f\circ\sigma=id_{\mathbb{S}^{2}}$. If there exists a lift of a positive
relation $t_{\alpha_{1}}t_{\alpha_{2}}\cdots t_{\alpha_{n}}=1$ in $\Mod_g$  to $\Mod_{g}^{k}$
such that $t_{\tilde{\alpha_{1}}}t_{\tilde{\alpha_{2}}}\cdots t_{\tilde{\alpha_{n}}}=t_{\delta_{1}}^{m_{1}}
t_{\delta_{2}}^{m_{2}}\cdots t_{\delta_{k}}^{m_{k}}$ where $m_i$'s are integers and $\delta_{i}$'s are
boundary curves, then the Lefschetz fibration $f:M\to \mathbb{S}^{2}$ admits $k$ disjoint sections
$S_1,\ldots,S_k$, where $S_j$ is of self-intersection $-m_j$ and vice versa ~\cite{bkm}.

  For  $i=1,2$, let $f_i :M_i \to \mathbb{S}^{2}$ be a genus-$g$ Lefschetz fibration with a regular
fiber $F_i$ and monodromy factorization $W_i=1$. Let  $r$ be an orientation-reversing
self-diffeomorphism of $S^{1} $ and $\phi : F_2\to F_1$ be an orientation-preserving
diffeomorphism. We remove a fibred neighborhood of $F_i$ from $M_i$ and glue the
resulting manifolds along their boundaries using the orientation-reversing diffeomorphism
$r\times \phi$. The resulting $4$-manifold is a genus-$g$ Lefschetz fibration over
$\mathbb{S}^{2}$ with monodromy factorization $W_1W_{2}^{\phi}$, which is called
a \textit{twisted fiber sum} of the Lefschetz fibrations $f_1$ and $f_2$. Moreover, for $i=1,2$, if the Lefschetz fibration $f_i :M_i \to \mathbb{S}^{2}$ admits a section with
self-intersection $m_i$, then the twisted fiber sum of $f_1$ and $f_2$ admits a section
with self-intersection $m_1+m_2$. Here the notation $W^{\phi}$ denotes the conjugated 
word of $W$, i.e., $W^{\phi}=t_{\phi(\alpha_{1})} t_{\phi(\alpha_{2})}\cdots t_{\phi(\alpha_{n})}$,
if $W=t_{\alpha_{1}}t_{\alpha_{2}}\cdots t_{\alpha_{n}}$.
\begin{figure}[h]
\begin{center}
\scalebox{0.35}{\includegraphics{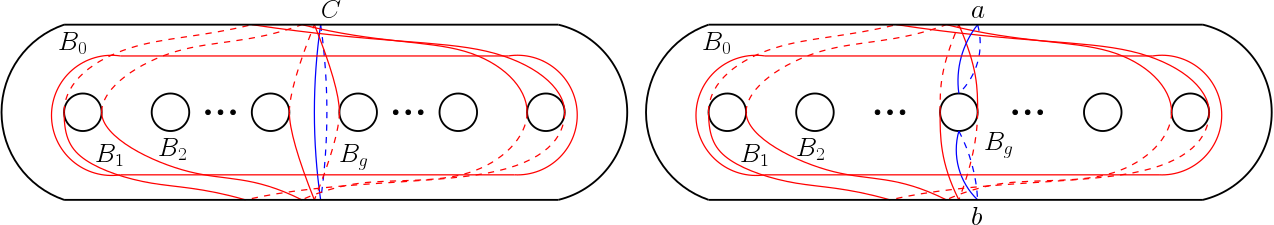}}
\caption{The curves $B_i$, $a$, $b$ and $C$ on $\Sigma_{g}$.}
\label{GM2}
\end{center}
\end{figure}
\begin{figure}[h]
\begin{center}
\scalebox{0.35}{\includegraphics{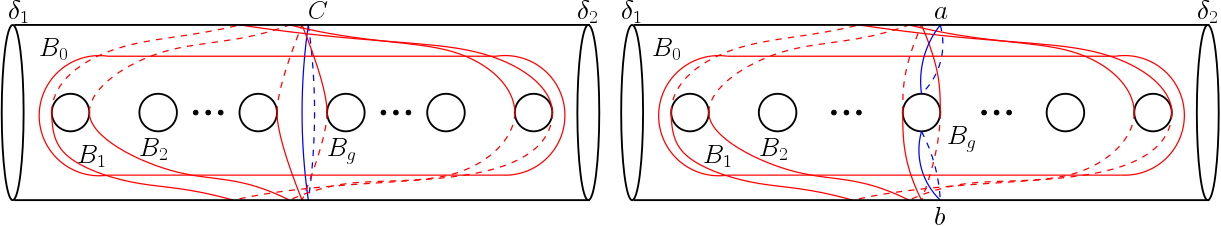}}
\caption{The curves $B_{i}$, $a$, $b$ and $C$ on $\Sigma_{g}^{2}$.}
\label{GM}
\end{center}
\end{figure}
\subsection{Generalized Matsumoto's relation}\label{s1}
Let $B_0,B_1,\ldots,B_g, a,b$ and $C$ be simple closed curves on $\Sigma_{g}$ as shown
in Figure ~\ref{GM2} and $W_g$ be the following word:

\[
W_g=\left\{\begin{array}{lll}
(t_{B_0}t_{B_1}\cdots t_{B_g}t_{C})^{2}& \textrm{if} & g=2k,\\
(t_{B_0}t_{B_1}\cdots t_{B_g}t_{a}^{2}t_{b}^{2})^{2}& \textrm{if} & g=2k+1.\\
\end{array}\right.
\]

The word $W_g$ represents the identity in the mapping class group $\Mod_{g}$,
which was shown by Matsumoto for $g=2$~\cite{mat} , and by Korkmaz~\cite{k1}
and independently by Cadavid~\cite{Cadavid} for $g\geq 3$ (see also ~\cite{dmp} for a different proof). Stipsicz and Ozbagci~\cite{os} showed
that $W_g=t_\delta$ in $\Mod_{g}^{1}$, where $\delta$ is the boundary of the genus-$g$ surface,  $\Sigma_{g}^{1}$. Korkmaz~\cite{kl} obtained a lifting of $W_g$ that equals to the product $t_{\delta_1}t_{\delta_2}$ in $\Mod_{g}^{2}$, where $\delta_1$ and $\delta_2$ are the boundary components of $\Sigma_{g}^{2}$. Recently, Hamada~\cite{h} gave a maximal set of disjoint $(-1)$ sections of $W_g$. One of his
liftings is $W_g=t_{\delta_1}t_{\delta_2}$, where the curves
$\delta_1$ and $\delta_2$ are as depicted in Figure~\ref{GM}.\par
Let $M_g$ be the total space of the Lefschetz fibration over $\mathbb{S}^{2}$ with the monodromy
factorization $W_g$. Then $M_g$ is diffeomorphic to
$\Sigma_{k}\times \mathbb{S}^{2}\# 4\overline{\C\P^{2}}$
(respectively $\Sigma_{k}\times \mathbb{S}^{2}\# 8\overline{\C\P^{2}}$)
if $g=2k$ (respectively $g=2k+1$)~\cite{k1,Cadavid}.

\subsection{Signature of a relation}\label{s2}
 Endo and Nagami~\cite{e11} gave a useful method to calculate the signature of a
Lefschetz fibration over $\mathbb{S}^{2}$ by introducing the notion of the signature
of a relation. This method allows one to determine the
signature of a Lefschetz fibration over $\mathbb{S}^{2}$ as the sum of signatures
of basic relations in its monodromy. They also explicitly compute the signature of
some known relations. Let us recall the definition of the signature of a relation and
the results that we will need later.

Let ${\mathcal F}$ be the free group generated by all isotopy classes of simple
closed curves on $\Sigma_{g}$. There is a natural homomorphism
$\varrho :\mathcal{F} \to \Mod_{g}$ mapping a simple closed curve $a$ to the Dehn twist $t_{a}$. Since $\Mod_{g}$ is finitely generated
by Dehn twists, the homomorphism $\varrho$ is surjective. An element of
$Ker \varrho$ is called a \textit{relator}. A relator $\rho$ is of the form
$\rho=c_{1}^{\epsilon_{1}}c_{2}^{\epsilon_{2}}\cdots c_{n}^{\epsilon_{n}}$,
where $c_i$'s are simple closed curves on $\Sigma_{g}$ and $\epsilon_{i}=\pm 1$
for $i=1,\ldots,n$. The word $\rho$ is said to be a  \textit{positive relator}
if $\epsilon_{i}=+1$ for all $i$. For instance, 
\[
L=x_1x_2x_3\delta_{1}^{-1}\delta_{2}^{-1}\delta_{3}^{-1}\delta_{4}^{-1}
\]
 is a relator coming from the lantern relation, which we call the \textit{lantern relator}, 
 where the curves $x_i$ and $\delta_i$ as shown in Figure~\ref{LR}.\\
 \noindent
  The words
 \[
(B_{0}B_{1} \cdots B_{g}C)^{2} \textrm{  if }  g \textrm{ is even,} 
\]
and
\[ 
(B_{0}B_{1} \cdots B_{g}a^{2}b^{2})^{2} \textrm{  if }  g \textrm{ is odd} 
\]
are also relators, where the curves $B_i,C,a$ and $b$ as shown in
Figure~\ref{GM2}.

%\begin{definition} 
For two relators $\rho=W_{1}^{-1}W_{2}$ and $\xi=UW_1V$  such that 
$U,V, W_1$ and $W_2$ are positive words in $\mathcal{F}$, one can obtain a new 
positive relator 
\[
\xi^{\prime}=\xi V^{-1}\rho V=UW_2V.
\]
This operation is called 
\textit{$\rho$-substitution to $\xi$}. When $\rho$ is a lantern relator then we 
say that $\xi^{\prime}$ is obtained by applying the lantern substitution to $\xi$.
%\end{definition}

There is an explicit homomorphism $c_{g}:Ker \varrho \to \mathbb{Z}$ inducing
the evaluation map $H_{2}(\Mod_{g})\to \mathbb{Z}$ for the cohomology class
of $\tau_{g}$, where $\tau_{g}:\Mod_{g}\times \Mod_{g}\to \mathbb{Z}$ is the
Meyer's signature cocycle. For a relator $\rho\in Ker\varrho$, \textit{the signature}
of $\rho$ is given by 
\[
I_{g}(\rho):=-c_{g}(\rho)-s(\rho),
\]
where $s(\rho)$ is the sum of the exponents of Dehn twists about separating
simple closed curves appearing in the word $\rho$. Moreover, Endo and Nagami
extended this definition to elements of the free group $\mathcal{F}$.

\ For the proofs of the following lemma and theorem, we refer the reader to~\cite{e11}.

\begin{lemma} \label{l1}The signature $I_{g}$ satisfies the following:
\item{(i)}\label{l1i}  $I_{g}(a)=-1$, where $a$ is the isotopy class of a separating curve.
\item{(ii)} \label{l1ii}$I_{g}(L)=+1$, where $L$ is a lantern relator.
\item{(iii)} \label{l1iii}$  I_{g}\big((B_{0}B_{1} \cdots B_{g}C)^{2}\big)=-4$ \hspace{0.3cm} if $g$ is even,\\
\hspace*{0.6cm} $I_{g}\big((B_{0}B_{1} \cdots B_{g}a^{2}b^{2})^{2}\big)=-8$ \hspace{0.3cm} if $g$ is odd.

\end{lemma}

\begin{theorem}\label{t11}
Let $f:X \to \mathbb{S}^{2}$ be a genus-$g$ Lefschetz fibration with the
monodromy $t_{c_1}t_{c_2}\cdots t_{c_n}$, so that  $c_{1}c_{2}\cdots c_{n} \in Ker(\varrho)$
is a positive relator. Then the signature of the $4$-manifold $X$ is equal
to the signature of $c_{1}c_{2}\cdots c_{n}$, i.e.,
\[
\sigma(X)=I_{g}(c_{1}c_{2}\cdots c_{n}).
\]
\end{theorem}

\subsection{Symplectic Fiber Sums} 
This subsection explains the symplectic fiber sum operation and two useful criteria to determine the minimality of a fiber sum.

 Let $X_{i}$ be a closed oriented smooth manifold of dimension $4$ containing a smoothly
embedded surface $\Sigma$ of genus $g\geq1$ such that the surface $\Sigma$ has zero
self-intersection in $X_i$ and represents a homology class of infinite order for each $i=1,2$.
 \textit{The generalized fiber sum} $X_1\#_{\varphi}X_2$ along closed embedded genus-$g$
surfaces $\Sigma$ is defined as
$(X_1\setminus  \upsilon \Sigma)\cup_{\varphi}(X_2 \setminus \upsilon \Sigma)$,
where $\upsilon \Sigma\cong\Sigma\times D^{2}$ denotes
a tubular neighbourhood of the surface $\Sigma$ in both $X_1$ and $X_2$ and the gluing map $\varphi$ is an
orientation-reversing and fiber preserving self-diffeomorphism of $S^{1}\times \Sigma$.

For symplectic manifolds $X_i$ and embedded symplectic subsurfaces in $X_i$ for each $i=1,2$, 
Gompf~\cite{g1} showed that the resulting manifold $X_1\#_{\varphi}X_2$ admits a symplectic
structure.
 
Let $e(X)$ be the Euler characteristic of a manifold $X$. For $X_1\#_{\varphi}X_2$, we have the following lemma.
 \begin{lemma}\label{lfib}
  Let $X_1\#_{\varphi}X_2$ be the fiber sum of $X_1$ and $X_2$ along closed
  embedded surface $\Sigma$ of genus $g$ ($g\geq1$) determined by $\varphi$. Then 
 \item{(i)} $e(X_1\#_{\varphi}X_2)=e(X_1)+ e(X_2)-2e(\Sigma),$
 \item{(ii)} $\sigma(X_1\#_{\varphi}X_2)=\sigma(X_1)+\sigma(X_2).$
 \end{lemma}
 
 One can determine the minimality of a symplectic fiber sum using the following theorem:
 
\begin{theorem}\cite{u,do}\label{t1}
Let $(X, \omega_{X})$ and $(Y, \omega_{Y})$ be two symplectic $4$-manifolds
containing an embedded symplectic surface $S$ of genus $g\geq 0$ and $M$ be the
symplectic fiber sum $X\#_{S}Y$. 
\begin{enumerate}
\item The $4$-manifold $M$ is not minimal if
\begin{enumerate}
\item $X \setminus S_{X}$ or $Y\setminus S_{Y}$ contains an embedded
symplectic $(-1)$-sphere, where $S_{X}\subset X$ and $S_{Y}\subset Y$ are
copies of the surface $S$,\\
or
\item $X\#_{S}Y=Z\#_{S_{\C\P^{2}}}\C\P^{2}$ with $S_{\C\P^{2}}$ 
an embedded $+4$-sphere in class 
$[S_{\C\P^{2}}]=2[H]\in H_{2}(\C\P^{2};\mathbb{Z})$ and $Z$ 
has at least two disjoint exceptional spheres $E_{i}$ so that  each  $E_i$ meets 
the submanifold $S_{Z}\subset Z$  transversely and positively in a single point with 
$[E_{i}]\cdot[S_{X}]=1$.
\end{enumerate}
\item If $X\#_{S}Y=Z\#_{S_{B}}B$, where $B$ is an $\mathbb{S}^{2}$-bundle over 
a surface of genus-$g$ and $S_{B}$ is a section
of this bundle, then $M$ is minimal if and only if $Z$ is minimal.
 \item $M$ is minimal in all other cases.
\end{enumerate}
\end{theorem}

The next proposition will also be useful for us
to determine the minimality of some Lefschetz fibrations constructed in this paper. 

\begin{proposition}\cite{bki}\label{p1}
Let $(X,f)$ be a Lefschetz fibration associated to a factorization
$W=W_{1}^{\phi}W_2$ in $\Mod_{g}$, where $\phi$ is any mapping
class and $W_1$, $W_2$ are positive factorizations of the identity in $\Mod_{g}$.
Then the  $4$-manifold $X$ is minimal.
\end{proposition}

%%%%%%%%%%%%%%%%%%%%%%%%%%%%%%%%%%%%%%%%%%%%%%%%

\section{Construction of  genus-$3$ Lefschetz fibrations}\label{S3}
In this section, we explicitly construct a positive factorization for  a genus-$3$
Lefschetz fibration over $\mathbb{S}^{2}$ whose total space is diffeomorphic to
$\mathbb{T}^{2}\times \mathbb{S}^{2}\# 6\overline{\C\P^{2}}$. To do this, we follow very closely the breeding technique and the same building blocks in~\cite{b1}. Note that this breeding technique is also used to construct the smallest
hyperelliptic genus-$3$ Lefschetz fibration by Baykur and Korkmaz~\cite{bk1} and to get holomorphic Lefschetz pencils on the four-torus~\cite{hh}. Following the positive factorization for a genus-$3$ fibration on $\mathbb{T}^{2}\times \mathbb{S}^{2}\# 6\overline{\C\P^{2}}$, we also construct some positive factorizations which yield simply-connected genus-$3$ Lefschetz fibrations over a sphere using the fiber sum operation and the lantern substitution.
\begin{figure}[h]
\begin{center}
\scalebox{0.35}{\includegraphics{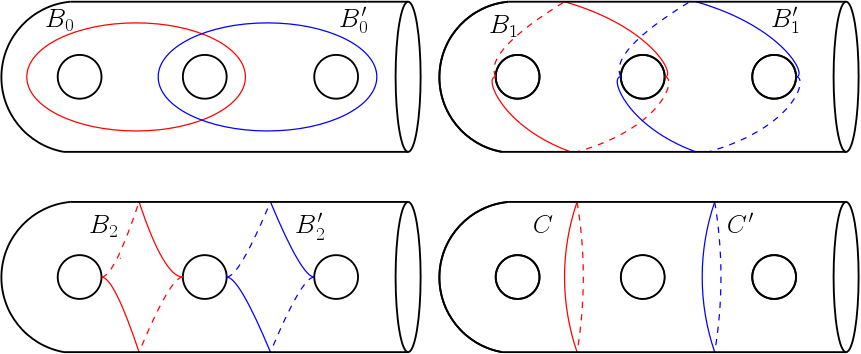}}
\caption{The curves $B_i$, $B_i^{\prime}$, $C$, $C^{\prime}$.}
\label{Con}
\end{center}
\end{figure}

\ Consider the genus-$3$ surface $\Sigma_{3}^{1}$ represented in Figure~\ref{Con}. 
The lifting of $W_2$ constructed by Hamada and the 
embeddings of the genus-$2$ surfaces $\Sigma_{2}^{1}$  and $\Sigma_{2}^{2}$
in $\Sigma_{3}^{1}$  
give rise to the following identities in $\Mod_{3}^{1}$:
\begin{eqnarray} \label{eq1}
(t_{B_{0}}t_{B_{1}}t_{B_{2}}t_{C})^{2}&=&(t_{C}t_{B_{0}}t_{B_{1}}
t_{B_{2}})^{2}=t_{C^{\prime}}, \\
\label{eq2}(t_{B_{0}^{\prime}}t_{B_{1}^{\prime}}t_{B_{2}^{\prime}}
t_{C^{\prime}})^{2}&=&t_{C}t_{\delta},
\end{eqnarray}
where the curves $B_i$, $B_{i}^{\prime}$, $C$ and $C^{\prime}$ are as shown 
in Figure~\ref{Con} and $\delta$ is a curve parallel to the boundary component 
of $\Sigma_{3}^{1}$. The first identity in (\ref{eq1}) comes from the commutativity of Dehn twists about disjoint curves $C$ and $C^{\prime}$ and the 
second identity in (\ref{eq1}) is obtained by capping off one boundary component in Hamada's lifting. 
Using the fact that  $t_{C}$ and $t_{C^{\prime}}$ commute, 
we get the following relation in $\Mod_{3}^{1}$:
\begin{eqnarray*}
t_{B_{0}}t_{B_{1}}t_{B_{2}}t_{C}t_{B_{0}}t_{B_{1}}t_{B_{2}}t_{B_{0}^{\prime}}
t_{B_{1}^{\prime}}t_{B_{2}^{\prime}}t_{C^{\prime}}t_{B_{0}^{\prime}}t_{B_{1}^{\prime}}
t_{B_{2}^{\prime}}t_{C^{\prime}}t_{C}t_{C}^{-1}t_{C^{\prime}}^{-1}=t_{\delta}.
\end{eqnarray*}

Finally, we get the following identity in $\Mod_{3}^{1}$:
\begin{eqnarray}\label{eq3}
t_{B_{0}}t_{B_{1}}t_{B_{2}}t_{C}t_{B_{0}}t_{B_{1}}t_{B_{2}}t_{B_{0}^{\prime}}
t_{B_{1}^{\prime}}t_{B_{2}^{\prime}}t_{C^{\prime}}
t_{B_{0}^{\prime}}t_{B_{1}^{\prime}}t_{B_{2}^{\prime}}=t_{\delta}
\end{eqnarray}
 consisting of the product
of positive Dehn twists about $12$ nonseparating curves $B_i$, $B_{i}^{\prime}$
and $2$ separating simple closed curves $C$ and $C^{\prime}$.
Let $W$ be the positive factorization of $t_{\delta}$ in (\ref{eq3}) and  
let $X$ denote the smooth $4$-manifold admitting the genus-$3$ Lefschetz fibration 
over $\mathbb{S}^{2}$ with the monodromy $W$. Observe that it admits a section of self-intersection $-1$.

For later use, we give an explicit presentation of $\pi_1(X)$ induced by the monodromy $W$.

Let  $a_i$ and $b_i$ be the generators of $\pi_{1}(\Sigma_{3})$ as 
shown in Figure~\ref{C1}.
By the theory of Lefschetz fibrations, since the Lefschetz fibration  $(X,f)$ with the monodromy $W$ has a section,
$\pi_{1}(X)$ is isomorphic to the quotient of $\pi_{1}(\Sigma_{3})$
by the normal subgroup generated by vanishing cycles of $(X,f)$ ~\cite{gs}.

We get the following relations coming from the vanishing cycles:
\begin{eqnarray}
 \label{eq4} B_{0}&=&b_{1}b_{2}=1,\\
 \label{eq5} B_{1}&=&a_{2}^{-1}[a_{3},b_{3}]b_{2}^{-1}b_{1}^{-1}a_{1}^{-1}=1,\\
 \label{eq6} B_{2}&=&a_{2}^{-1}[a_{1},b_{1}^{-1}]a_{1}^{-1}=1,\\
 \label{eq7} B_{0}^{\prime}&=&b_{2}b_{3}=1,\\
 \label{eq8}B_{1}^{\prime}&=&a_{3}^{-1}b_{3}^{-1}b_{2}^{-1}a_{2}^{-1}=1,\\
\label{eq9} B_{2}^{\prime}&=&b_{3}a_{3}^{-1}b_{3}^{-1}a_{2}^{-1}=1,\\
 \label{eq10} C&=&[a_{1},b_{1}]=1,\\
 \label{eq11}C^{\prime}&=&[a_{3},b_{3}]=1.
\end{eqnarray}
 In addition, $\pi_{1}(X)$ has the relation
 \begin{eqnarray}
  \label{pi3} b_{3}^{-1}b_{2}^{-1}b_{1}^{-1}(a_1b_1a_{1}^{-1})
 (a_2b_2a_{2}^{-1})(a_3b_3a_{3}^{-1})=1.
\end{eqnarray} 
Thus, $\pi_{1}(X)$ admits a presentation with generators
$a_1,a_2,a_3,b_1,b_2,b_3$ and with the relations
$(\ref{eq4})-(\ref{pi3})$.

The relations (\ref{eq4}) and (\ref{eq7}) give $b_{1}=b_{2}^{-1}=b_{3}$.
It follows from the relations (\ref{eq6}), (\ref{eq7}), (\ref{eq8})
and (\ref{eq10}), we obtain $a_{1}=a_{2}^{-1}=a_{3}$.
We conclude that $\pi_{1}(X)$ is a free abelian group of rank $2$
generated by $a_{1}$ and $b_{1}$.\\

\begin{theorem}\label{t6} 
The $4$-manifold $X$ is diffeomorphic to
$\mathbb{T}^{2}\times \mathbb{S}^{2}\# 6\overline{\C\P^{2}}$.
\end{theorem}

\begin{proof}
We first compute $\sigma(X)$, the signature of $X$, using 
Endo and Nagami's method given in~\cite{e11}.
(Alternatively, $\sigma(X)$ can be calculated using Ozbagci's algorithm~\cite{oz}).

Consider the relator 
$(B_{0}B_{1} B_{2} C)^2 ( B_{0}^{\prime} B_{1}^{\prime}B_{2}^{\prime}C^{\prime})^2 C^{-1} C^{\prime-1}$
associated to the relation 
\[
W=t_{B_{0}}t_{B_{1}}t_{B_{2}}t_{C}t_{B_{0}}t_{B_{1}}t_{B_{2}}t_{B_{0}^{\prime}}
t_{B_{1}^{\prime}}t_{B_{2}^{\prime}}t_{C^{\prime}}t_{B_{0}^{\prime}}
t_{B_{1}^{\prime}}t_{B_{2}^{\prime}}=1
\] 
in $\Mod_{3}$ obtained by the factorization of $t_{\delta}$ in 
$\Mod_{3}^{1}$  by capping off the boundary component. 
It follows from Theorem~\ref{t11}, the additivity of the signature
and  Lemma~\ref{l1} that we have,
 \begin{eqnarray*}
 \sigma(X)& =& I_{3}\big((B_{0}B_{1}B_{2}C)^{2}
 (B_{0}^{\prime}B_{1}^{\prime}B_{2}^{\prime}C^{\prime})^{2}C^{-1}
 (C^{\prime})^{-1}\big),\\
& = &I_{3}\big((B_{0}B_{1}B_{2}C)^{2}\big)+I_{3}\big((B_{0}^{\prime}B_{1}^{\prime}
B_{2}^{\prime}C^{\prime})^{2}\big)-I_{3}(C)-I_{3}(C^{\prime}),\\
 & = & -4-4-(-1)-(-1)=-6.
 \end{eqnarray*}
The topological invariants of $e(X)$ and $c_{1}^{2}(X)$ are computed as
\begin{eqnarray*}
 e(X)&=&e(\mathbb{S}^{2})e(\Sigma_{3})+\# \rm{singular \hspace{0.12cm} fibers}
 =2(-4)+14
 =6,\\
 c_{1}^{2}(X)&=&2e(X)+3\sigma(X)
 =2(6)+3(-6)7
 =-6.
 \end{eqnarray*}

We now prove that $X$ is a ruled surface. Suppose that $X$ is neither 
rational nor ruled. Let $\widetilde{X}$ be the minimal model of $X$,
so $X\cong\widetilde{X}\#k\overline{\C\P^{2}}$ for some
non-negative integer $k$. It is easy to see that 
\begin{eqnarray*}
c_{1}^{2}(\widetilde{X})=c_{1}^{2}(X)+k=-6+k.
 \end{eqnarray*}
Since $\widetilde{X}$ is a minimal symplectic $4$-manifold that
is neither rational nor ruled, $c_{1}^{2}(\widetilde{X})\geq0$ by
Liu~\cite{liu2} and Taubes~\cite{t2}, so that we have $k\geq6$. 
Moreover, since $X$ has $k\geq 6$ disjoint exceptional spheres, 
it follows from ~\cite[Theorem 4.7]{sa1} that $k\leq2g-2=4$,
which is a contradiction. Therefore, $X$ is either a rational or a ruled surface. 
Since $b_{1}(X)=2$, we conclude that $X$ is diffeomorphic to
(a blow up of) 
a ruled surface with $(b_{2}^{+},b_{2}^{-},b_{1})=(1,7,2)$. Thus, we conclude that $X$ is diffeomorphic to
$\mathbb{T}^{2}\times \mathbb{S}^{2}\# 6\overline{\C\P^{2}}$.
 \end{proof}
 
\begin{figure}[t]
 \scalebox{0.35}{\includegraphics{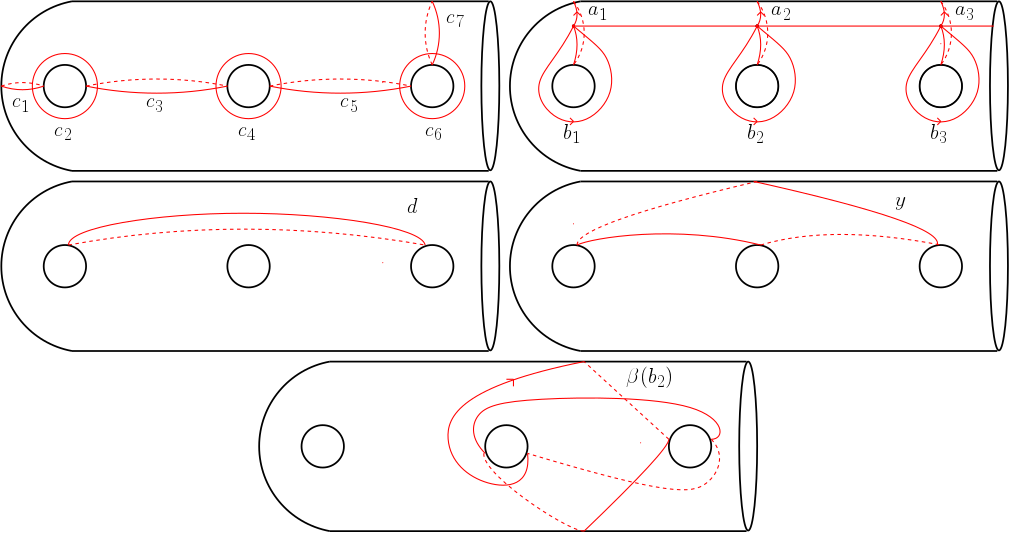}}
 \caption{The curves $c_i$, the generators of $\pi_1(\Sigma_3)$, $d$, $y$ and $\beta(b_2)$.} \label{C1}
 \end{figure}

%%%%%%%%%%%%%

\subsection{Construction of genus-$3$ Lefschetz fibrations }
In order to construct further genus-$3$ Lefschetz fibrations,  let us rewrite $W$ as 
\[
W=Vt_{B_2^{\prime}}^{2}t_{B_2}^{2},
\]
where 
\[
V=t_{t_{_{B_2}}^{-2}(B_0)}t_{t_{_{B_2}}^{-2}(B_1)}t_{t_{_{B_2}}^{-1}(C)}
t_{t_{_{B_2}}^{-1}(B_0)}t_{t_{_{B_2}}^{-1}(B_1)}t_{B_0^{\prime}}t_{B_1^{\prime}}
t_{t_{_{B_2}}^{\prime}(C^{\prime})}t_{t_{_{B_2^{\prime}}}(B_0^{\prime})}
t_{t_{_{B_2^{\prime}}}(B_1^{\prime})}.
\]

Let
$\alpha=t_{c_{4}}t_{c_3}t_{B_{2}^{\prime}}t_{c_{4}}t_{c_{2}}t_{c_{1}}t_{B_{2}}t_{c_{2}}$
and 
$\beta=  t_{c_{2}}t_{c_{4}}t_{c_5}t_{B_{2}}t_{c_{4}}t_{c_{6}}t_{c_{7}}
 t_{B_{2}^{\prime}}t_{c_{6}}$,
where the curves $c_i$'s are as in Figure~\ref{C1}. It is easy to see that 
$\alpha(B_{2}^{\prime})=c_3$, $\alpha(B_{2})=c_1$,
$\beta(B_{2}^{\prime})=c_7$ and $\beta(B_{2})=c_5$. The conjugations of $W$ with $\alpha$ and $\beta$ give the factorizations 
\begin{eqnarray*}
t_{\delta}=W^{\alpha}=V^{\alpha}t_{\alpha(B_{2}^{\prime})}^{2}t_{\alpha(B_{2})}^{2}=V^{\alpha}t_{c_3}^{2}t_{c_{1}}^{2}
\end{eqnarray*}
and 
\begin{eqnarray*}
t_{\delta}=W^{\beta}=V^{\beta}t_{\beta(B_{2}^{\prime})}^{2}
t_{\beta(B_{2})}^{2}=V^{\beta}t_{c_7}^{2}t_{c_5}^{2}=t_{c_7}^{2}t_{c_5}^{2}V^{\beta}.
\end{eqnarray*}

It follows that
 \begin{eqnarray}\label{xx1}
	t_{\delta}^{2}=W^{\alpha}W^{\beta}=V^{\alpha}t_{c_1}^{2}
	t_{c_{3}}^{2}t_{c_5}^{2}t_{c_{7}}^{2}V^{\beta}=V^{\beta}V^{\alpha}t_{c_1}^{2}t_{c_3}^{2}t_{c_{5}}^{2}t_{c_{7}}^{2}. 
\end{eqnarray}

We see that the curves $\{ c_1,c_3,c_5,c_7\}$ bound 
a sphere with four boundary components, which allows us to use 
the lantern substitution.  
Using the lantern relation $t_{c_1}t_{c_3}t_{c_5}t_{c_7}=t_{d}t_{y}t_{a}$, where the curves $d$ 
and $y$ as in Figure ~\ref{C1} and the curve $a$ is depicted in Figure ~\ref{GM} for $g=3$,  we get the identity 
\begin{eqnarray}
V^{\beta}V^{\alpha}t_{d}t_{y}t_{a}t_{c_1}t_{c_3}t_{c_{5}}
t_{c_{7}}=t_{\delta}^{2}\label{xx2}
\end{eqnarray}
in $\Mod_{3}^{1}$. The identity (\ref{xx2}) still includes the Dehn twist about the curves $\{ c_1,c_3,c_5,c_7\}$, which allows us to 
use one more lantern substitution, then we get the following identity
\begin{eqnarray}
V^{\beta}V^{\alpha}(t_{d}t_{y}t_{a})^2=t_{\delta}^{2}.\label{xx3}
\end{eqnarray}

Let
\[
 \begin{array}{ll}
 
 \U_{1}=\alpha(t_{B_2}^{-2}(B_0))   & \hspace*{2cm} 
 \U'_{1}=\beta (t_{B_2}^{-2}(B_0)) \\
 \U_{2}=\alpha(t_{B_{2}}^{-2}(B_{1}))& \hspace*{2cm}
 \U'_{2}=\beta (t_{B_{2}}^{-2}(B_{1}))\\
 \U_{3}=\alpha(t_{B_2}^{-1}(C)) &\hspace*{2cm}
 \U'_{3}=\beta (t_{B_2}^{-1}(C))\\
 \U_{4}=\alpha(t_{B_2}^{-1}(B_0))&\hspace*{2cm}
 \U'_{4}=\beta (t_{B_2}^{-1}(B_0))\\
 \U_{5}=\alpha(t_{B_2}^{-1}(B_1))&\hspace*{2cm}
 \U'_{5}=\beta (t_{B_2}^{-1}(B_1))\\
 \U_{6}=\alpha(B_0^{\prime})&\hspace*{2cm}
 \U'_{6}=\beta(B_0^{\prime})\\
 \U_{7}=\alpha(B_1^{\prime})&\hspace*{2cm}
 \U'_{7}=\beta (B_1^{\prime})\\
 \U_{8}=\alpha(t_{B_{2}^{\prime}}(C^{\prime}))&\hspace*{2cm}
 \U'_{8}=\beta (t_{B_{2}^{\prime}}(C^{\prime}))\\
 \U_{9}=\alpha(t_{{B_{2}^{\prime}}}(B_0^{\prime}))&\hspace*{2cm}
 \U'_{9}=\beta (t_{B_{2}^{\prime}}(B_0^{\prime}))\\
 \U_{10}=\alpha(t_{B_{2}^{\prime}}(B_1^{\prime}))&\hspace*{2cm}
 \U'_{10}=\beta (t_{B_{2}^{\prime}}(B_1^{\prime}))
 
 \end{array} 
\]
so that (\ref{xx1}), (\ref{xx2}) and (\ref{xx3}) are given, respectively, as
\begin{eqnarray}
	t_{\delta}^{2} &=&  t_{\U'_{1}}t_{\U'_{2}}\cdots t_{\U'_{10}}t_{\U_{1}}
           t_{\U_{2}}\cdots t_{\U_{10}}t_{c_1}^2t_{c_3}^2t_{c_{5}}^2
t_{c_{7}}^2.   \label{x1} \\
t_{\delta}^{2} &=&  t_{\U'_{1}}t_{\U'_{2}}\cdots t_{\U'_{10}}t_{\U_{1}}
          t_{\U_{2}}\cdots t_{\U_{10}} (t_{d}t_{y}t_{a})t_{c_1}t_{c_3}t_{c_{5}}t_{c_{7}}.  \label{x2} \\
  t_{\delta}^{2} &=&  t_{\U'_{1}}t_{\U'_{2}}\cdots t_{\U'_{10}}t_{\U_{1}}
          t_{\U_{2}}\cdots t_{\U_{10}} (t_{d}t_{y}t_{a})^2.    \label{x3}        
\end{eqnarray}

Let $(X_1,f_1)$, $(X_2,f_2)$ and $(X_3,f_3)$  be the genus-$3$ Lefschetz fibrations with the monodromies (\ref{x1}),(\ref{x2}) and (\ref{x3}), respectively.

For later use, up to conjugation and the inversion, 
we write some of the vanishing cycles of $X_1$, $X_2$ and $X_3$ in the fundamental group of $\Sigma_3$, $\pi_1(\Sigma_3)$, with the generating set $\lbrace a_1,a_2,a_3,b_1,b_2,b_3\rbrace$.  

The following vanishing cycles $\U_{i}$ are shown in Figure~\ref{U10}. One may find that

%%%%%%%%%%%%%%%%%%%%%
 \begin{align}
  &\label{eq12}\U_{4}=a_1^{-1}a_2b_1^{-1}a_1b_1a_2^{-1}b_3a_3^{-1}
 b_3^{-1}b_2^{-1}a_2^{-1}a_1a_2b_1^{-1}a_1b_1a_2^{-1}b_3a_3^{-1}b_3^{-1}b_2^{-1}a_2^{-1}a_1\\
      &	\hspace*{6.5cm}
      a_2b_1^{-1}a_1^{-2}a_2b_2b_3a_3^{-1}
 b_3^{-1}a_2^{2}b_1^{-1}a_1, \nonumber \\
& \label{eq13}\U_{4}^{\prime}=a_1b_1^{-1}b_2[b_3,a_3]a_2b_2^{-1}b_3^{-1}a_3^{-1}a_2b_2^{-1}a_2b_2^{-1}b_1b_2[b_3,a_3]a_2b_2^{-1}a_3^{-1}a_2b_2^{-1},\\
     &\label{eq16} \U_{6}=a_1^{-1}a_2b_2^{-1}b_1^{-1}a_1b_1b_3a_2
     b_2^{-1}b_3a_3b_3^{-1}a_2b_2^{-1}b_1^{-1},\\  
  &\label{eq14} \U_{7}=a_1^{-1}a_2b_2^{-1}b_1^{-1}a_1b_1b_3a_3a_2b_2^{-1}a_2^{-1}a_1a_2b_2^{-1}b_1^{-1},\\
&\label{eq15}\U_{8}=a_1^{-1}a_2b_2a_2^{-1}a_3^{-1}a_2b_2^{-1}a_2^{-1}a_1b_3a_3^{-1}b_3^{-1}.
 \end{align}
\begin{figure}[h]
\begin{center}
\scalebox{0.20}{\includegraphics{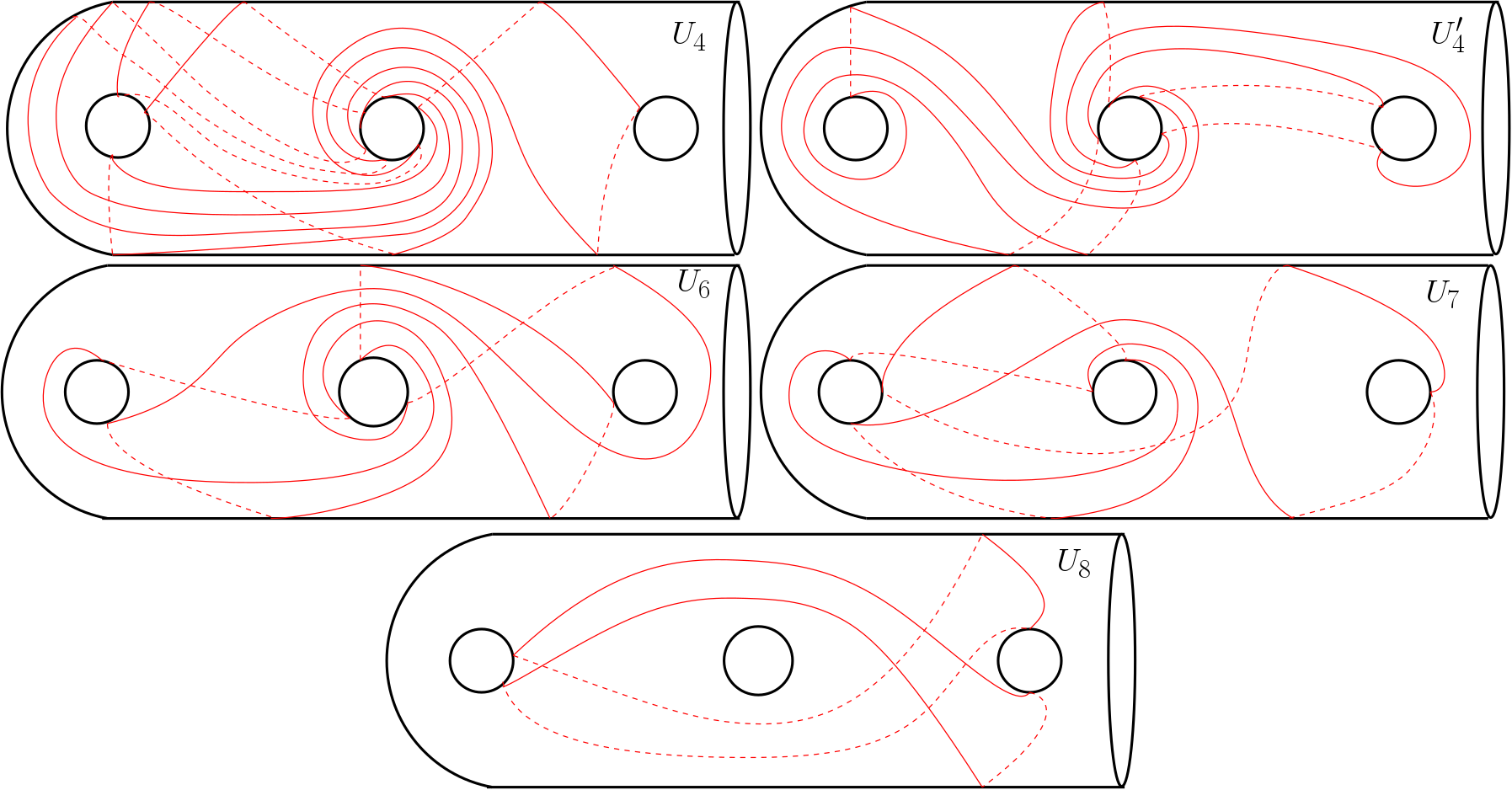}}
\caption{The curves $\U_i$'s for $i=4,6,7,8$ and $\U_4^{\prime}$} \label{U10}
\end{center}
\end{figure}

We now prove that for each $i=1,2,3$, the fundamental group $\pi_{1}(X_i)$ is trivial.

\begin{lemma} \label{pix1} 
 The $4$ manifold $X_1$ is simply connected.
\end{lemma}
\begin{proof}
The monodromy of $(X_1,f_1)$ is given in (\ref{x1}). Since this Lefschetz fibration has a section,
$\pi_{1}(X_1)$ has a presentation with generators $a_1,a_2,a_3,b_1,b_2,b_
3$ and with the 
defining relations
 \begin{eqnarray*}
  b_{3}^{-1}b_{2}^{-1}b_{1}^{-1}(a_1b_1a_{1}^{-1})(a_2b_2a_{2}^{-1})(a_3b_3a_{3}^{-1})=1, \\
  \U_i'=\U_i=c_1=c_3=c_5=c_7=1, \ \ (i=1,2, \ldots,10).
\end{eqnarray*}

Note that $c_1=a_1=1$, $c_3=a_1a_2^{-1}=1$, $c_5=a_2a_3^{-1}=1$ and  $c_7=a_3=1$ give 
\[
a_1=a_2=a_3=1
\]
 in $\pi_{1}(X_1)$.
It follows from the identity (\ref{eq12}) that the relation $\U_4=1$ gives  
\begin{eqnarray}
 b_2^{2}=b_1^{-1}b_2^{-1}b_1^{-1}. \label{rel1}
\end{eqnarray}
By the identity (\ref{eq14}), we get 
\begin{eqnarray}
 b_3=b_2b_1b_2^{2}. \label{rel2}
\end{eqnarray}
The relations (\ref{rel1}) and (\ref{rel2}) imply that $b_3=b_1^{-1}$.\\
Using the relation $\U_{4}^{\prime}=1$, the identity (\ref{eq13}) gives the following relation
\begin{eqnarray}
 b_1^{-1}b_3^{-1}b_2^{-2}b_1b_2^{-1}=1 \nonumber.
\end{eqnarray}
It follows from $b_3=b_1^{-1}$ that we have 
\begin{eqnarray}
 b_2^{-2}b_1b_2^{-1}=1\nonumber,
\end{eqnarray}
which implies that $b_1=b_2^{3}$.
Therefore $b_1$ and $b_3$ can be written in terms of  $b_2$. 
It follows that $\pi_1(X_1)$ is abelian and is isomorphic to a quotient of 
the free abelian group $\mathbb{Z}$, generated by
$b_2$. Hence, $\pi_1(X_1)$ is isomorphic to $H_1(X_1;\mathbb{Z})$.

We will denote the homology class of the curves by the same letters as denoted by the curves. 
Let us now determine the homology class of the vanishing cycle $\U_6^{\prime}=\beta(B_0^{\prime})$. 
  It follows from the homology class of $B_0^{\prime}=b_2+b_3$ that we have
  \begin{align*}
 \U_6^{\prime}=\beta(B_0^{\prime})&=\beta(b_2+b_3). 
  \end{align*}
By abelianizing above relations, we get the relations 
\begin{align*}
b_1=-b_3=3b_2
\end{align*}
 in $H_1(X_1;\mathbb{Z})$.
 Using the homology class of $\beta(b_2)$, which is given by
\begin{align*}
\beta(b_2)=2a_2-2b_2-b_3
\end{align*}
as shown in Figure ~\ref{C1}, one can obtain that
\begin{align*}
 \U_6^{\prime}&=\beta(b_2+b_3)=\beta(-2b_2)=-4a_2+4b_2+2b_3=-4a_2-2b_2
 =-2b_2=0,
  \end{align*}
since the homology classes of $a_2$ is trivial and $b_3=-3b_2$.

 It follows from the 
identity $b_1=3b_2$ that the relation $2b_2=0$ implies that 
\[
b_1=-b_3=b_2.
\]
The abelianization of (\ref{rel1}) gives $b_2=-2b_1$ implying that $3b_2=0$. Hence, $b_2$ is trivial  since $2b_2=0.$
 This completes our claim 
that $H_{1}(X_{1};\mathbb{Z})=0$.
 \end{proof}
\begin{lemma}\label{pix2}
 The $4$-manifold  $X_2$ is simply connected.
 \end{lemma}
 \begin{proof}
 The monodromy of $(X_2,f_2)$ is given in (\ref{x2}). Since the Lefschetz fibration
 $(X_2,f_2)$ admits a section, $\pi_{1}(X_2)$ has a presentation with generators
 $a_1,a_2,a_3,b_1,b_2,b_3$ and with defining relations
 \[
  b_{3}^{-1}b_{2}^{-1}b_{1}^{-1}(a_1b_1a_{1}^{-1})
  (a_2b_2a_{2}^{-1})(a_3b_3a_{3}^{-1})=1, 
  \]
  \[
  U_i'=U_i=c_1=c_3=c_5=c_7=a=d=y=1, \ \ i=1,2, \ldots,10.
\]
The fundamental group $\pi_{1}(X_2)$ has the same relations as 
$\pi_{1}(X_1)$ and some additional relations. This immediately shows that $~\pi_{1}(X_2)=1$.
\end{proof}
 \begin{lemma}\label{pix3}
 The $4$-manifold  $X_3$ is simply connected.
 \end{lemma}
 
 \begin{proof}
  The monodromy of $(X_3,f_3)$ is given in (\ref{x3}). Since this Lefschetz fibration
 has a section, $\pi_{1}(X_3)$ has a presentation with generators
 $a_1,a_2,a_3,b_1$, $b_2,b_3$ and with defining relations
 \[
  b_{3}^{-1}b_{2}^{-1}b_{1}^{-1}(a_1b_1a_{1}^{-1})
  (a_2b_2a_{2}^{-1})(a_3b_3a_{3}^{-1})=1, 
  \]
  \[
  U_i'=U_i=c_1=a=d=y=1, \ \ i=1,2, \ldots,10.
\]
The relations $a=a_2=1$, $d=a_1a_3^{-1}=1$ and $y=a_1a_2^{-1}a_3=1$ imply that $a_1=a_3=a_{3}^{-1}$ in $\pi_{1}(X_3)$. One can conclude that $a_2=a_3^{2}=1$ and $a_1=a_3$.

Let us consider the relations $U_6=1$ and $U_7=1$. The identities (\ref{eq16})
 and (\ref{eq14}) turn out to be the following identity
\begin{align}
[b_3,a_1]=b_2a_1b_2^{-1}.\label{bc}
\end{align}
Combining the relation $U_8=1$ with its presentation (\ref{eq15}) and the identity (\ref{bc}), one can conclude that $a_1=1$. This also shows that $a_2=a_3=1$. Hence, $\pi_{1}(X_3)$ has the same relations as $\pi_{1}(X_1)$. The remaining part of the proof follows as in
the proof of Lemma~\ref{pix1}.
\end{proof}
 \begin{theorem}
  \label{t3}The genus-$3$ Lefschetz fibration $f_{i}:X_{i}\to \mathbb{S}^{2}$
  is minimal and has  
 \item{(i)} $e(X_{i})=21-i,$
\item{(ii)} $c_{1}^{2}(X_{i})=3+i,$
\item{(iii)} $\pi_{1}(X_{i})=1,$
for $i=1,2,3$.
\end{theorem}
\begin{proof}
For each $i=1,2,3$, the Euler characteristic $e(X_i)$ of $X_i$, is given by
\begin{eqnarray*}
e(X_i)=e(\mathbb{S}^{2})e(\Sigma_{3})+\# \rm{singular \hspace{0.12cm} fibers}
=2(-4)+29-i=21-i
\end{eqnarray*}
and the signature $\sigma(X_i)$ of $X_i$ is 
\begin{eqnarray*}
\sigma(X_i)=\sigma(X_1)+\sigma(X_1)+i-1=-13+i
\end{eqnarray*}
using Novikov additivity and the fact that the lantern substitution increases
the signature by $1$ from Lemma~\ref{l1}(ii). The topological invariant  $c_1^{2}(X_i)$
is given by
\begin{eqnarray*}
c_{1}^{2}(X_i)=2e(X_i)+3\sigma(X_i)=3+i.
\end{eqnarray*}
By Proposition~\ref{p1}, $X_1$ is minimal. One can explain the minimality
of $X_2$ by considering a lantern substitution as a rational blowdown surgery along
a $-4$ sphere~\cite{f1}. The rational blowdown surgery along a $-4$ sphere can be
obtained by the symplectic sum operation. Hence,
$X_2=X_1\#_{S,V_{\C\P^{2}}}\C\P^{2}$ where $S$
is symplectic $-4$ sphere in $X_1$ and $V_{\C\P^{2}}$ is an embedded $+4$ sphere in $\C\P^{2}$ in the class of
$[V_{\C\P^{2}}]=2[H]\in H_2(\C\P^{2};\mathbb{Z})$.
Since $X_1$ is minimal, it follows from Theorem~\ref{t1} that $X_2$ is minimal. 
Similarly, since $X_3$ can be viewed as a symplectic sum of minimal $X_2$ and $\C\P^{2}$, $X_3$ is minimal.

Finally,  $\pi_{1}(X_i)$ is trivial for $i=1,2,3$ by Lemma~\ref{pix1},~\ref{pix2} and~\ref{pix3}.
\end{proof}

%%%%%%%%%%%%%%%%%%%%%%%%%%%%%%%%%%%%%%%%%%%%%%%%%%%%%%%%%%%%%%%%%%%%%%%%%%%%%%%%%%%%%%%%%%%%%%%%%%%%%%%%%%%%%%%%%%%%

\section{Exotic fibered $4$-manifolds with $b_2^{+}=3$}\label{S4}

In this section, we construct genus-$3$ Lefschetz fibrations over $\mathbb{S}^{2}$, whose total spaces are minimal symplectic $4$-manifolds homeomorphic but not diffeomorphic to $3\C\P^{2} \#  k\overline{\C\P^{2}}$ for $k=13,\ldots ,19$. In our construction, we use the Lefschetz fibration $f:X\to \mathbb{S}^{2}$ prescribed by the factorization $W$ and the generalized  Matsumoto Lefschetz fibration for $g=3$.\par
Consider the generalized Matsumoto's Lefschetz fibration for $g=3$ so that its total space is $M_3$ with the monodromy $W_{3}=(t_{\beta_{0}}
t_{\beta_{1}}t_{\beta_{2}}t_{\beta_{3}}t_{a}^{2}t_{b}^{2})^{2}$.
Here we denote the vanishing cycles of $M_3$ by $\beta_i$'s instead of $B_i$'s
as shown in Figure~\ref{GM} to distinguish them from the vanishing cycles of the Lefschetz fibration $f:X\to \mathbb{S}^{2}$.
To determine the relations in the fundamental group of $M_3$, $\pi_{1}(M_3)$, consider the following
identification of $\pi_{1}(M_3)$ using the existence of
sections of Matsumoto's fibrations :
\begin{eqnarray*}
 \pi_{1}(M_{3})=\pi_{1}(\Sigma_{3})/\langle\beta_0,\beta_1,\beta_2,\beta_3,a,b\rangle,
 \end{eqnarray*}
 \begin{eqnarray}
 \label{eq65}
\beta_0&=&b_1b_2b_3=1,\\
\label{eq66}
\beta_1&=&b_1b_2b_3a_3a_1=1,\\
 \label{eq67}
\beta_2&=&b_2b_3a_3b_{3}^{-1}a_1=1,\\
 \label{eq68}
\beta_3&=&a_2b_2[b_3,a_3]a_2=1,\\
 \label{eq69}
a&=&a_2=1,\\
 \label{eq70}
b&=&[a_1,b_{1}^{-1}]a_{2}^{-1}=1.
\end{eqnarray}
The equations (\ref{eq65}) and (\ref{eq66}) yield  the relation $a_3a_1=1$ . Note that in $\pi_{1}(M_3)$
\[
 b_{3}^{-1}b_{2}^{-1}b_{1}^{-1}(a_1b_1a_{1}^{-1})(a_2b_2a_{2}^{-1})(a_3b_3a_{3}^{-1})=1.
\] 
This implies that $[b_3,a_3]=1$ using the relations (\ref{eq69}) and (\ref{eq70}). Hence, the relation (\ref{eq68}) yield $b_2=1$. Finally, the relations
$ a_3a_1=1$ and $b_1b_3=1$ hold in $\pi_{1}(M_3)$ implying that $\pi_{1}(M_3)$ is the free abelian group of rank $2$.\par
We first present the minimal symplectic genus-$3$ Lefschetz fibrations with
$(b_{2}^{+},b_{2}^{-})=(3,19)$ and $(3,18)$. To obtain such Lefschetz fibrations,
consider the following positive factorization:
 \begin{align}
W_3W_{3}^{\phi}=(t_{\beta_{0}}t_{\beta_{1}}t_{\beta_{2}}t_{\beta_{3}}
t_{a}^{2}t_{b}^{2})^{2}(t_{\phi(\beta_{0})}t_{\phi(\beta_{1})}t_{\phi(\beta_{2})}
t_{\phi(\beta_{3})}t_{\phi(a)}^{2}t_{\phi(b)}^{2})^{2} =t_{\delta}^{2},\label{m19}
\end{align}
where $\phi=t_{b_3}t_{\beta_0}t_{c_1}$. 
%%%%%%%%%%%%

Let us rewrite the positive factorization $W_3$. In \cite{k1}, it is shown that the
product of positive Dehn twists 
$t_{\beta_{0}}t_{\beta_{1}}t_{\beta_{2}}t_{\beta_{3}}t_{a}^{2}t_{b}^{2}$
is the vertical involution $\iota$ of the genus-$3$ surface with two fixed points.
Hence, it preserves the curve $\beta_{0}$, then we have 
\[
t_{\beta_{0}}t_{\beta_{1}}t_{\beta_{2}}t_{\beta_{3}}t_{a}^{2}t_{b}^{2}(\beta_{0})=\beta_{0}.
\]
By applying Lemma \ref{le}, we get the following identity of the factorization $W_3$:
\begin{eqnarray*}
W_3&=&t_{\beta_{0}}t_{\beta_{1}}t_{\beta_{2}}t_{\beta_{3}}t_{a}^{2}t_{b}^{2}
t_{\beta_{0}}t_{\beta_{1}}t_{\beta_{2}}t_{\beta_{3}}t_{a}^{2}t_{b}^{2}\\
&=&t_{\beta_{0}}t_{\beta_{0}}t_{\beta_{1}}t_{\beta_{2}}t_{\beta_{3}}
t_{a}^{2}t_{b}^{2}t_{\beta_{1}}t_{\beta_{2}}t_{\beta_{3}}t_{a}^{2}t_{b}^{2}\\
&=&t_{\beta_{0}}^{2}(t_{\beta_{1}}t_{\beta_{2}}t_{\beta_{3}}t_{a}^{2}
t_{b}^{2})^{2}=t_{\delta}.
\end{eqnarray*}
It is easy to see that $\phi(\beta_{0})=c_1$. Then, we get  
\begin{eqnarray}
W_3W_{3}^{\phi}=T_1 t_{a}t_{b}t_{c_{1}}^{2}T_2=t_{\delta}^{2},\label{mm18}
\end{eqnarray}
where 
$T_1=t_{\beta_{0}}t_{\beta_{1}}t_{\beta_{2}}t_{\beta_{3}}t_{a}^{2}
t_{b}^{2}t_{\beta_{0}}t_{\beta_{1}}t_{\beta_{2}}t_{\beta_{3}} t_{a}t_{b}$
and
$T_2=(t_{\phi(\beta_{1})}t_{\phi(\beta_{2})}t_{\phi(\beta_{3})}
t_{\phi(a)}^{2}t_{\phi(b)}^{2})^{2}$.
Since the curves $\lbrace a,b,c_1,c_1\rbrace$ bound a sphere with
four holes, we can use the lantern relation $t_at_bt_{c_1}^{2}=t_{c_3}t_Ct_{B_2}$
to get the identity
\begin{eqnarray}
W_3W_{3}^{\phi}=T_1t_{c_3}t_Ct_{B_2} T_2=t_{\delta}^{2}.\label{m18}
\end{eqnarray}
Let $M_{19}$ and $M_{18}$ be the genus-$3$ Lefschetz fibrations with the monodromies
(\ref{m19}) and (\ref{m18}), respectively. We now prove that
$\pi_{1}(M_{19})=\pi_{1}(M_{18})=~1$.
\begin{lemma}\label{pim19}
The $4$-manifold  $M_{19}$ is simply-connected.
\end{lemma}
\begin{proof}
Since the Lefschetz fibration $M_{19}$ has a section, $\pi_{1}(M_{19})$
has a presentation with the generators $a_j,b_j$, $(j=1,2,3)$ and with the relations 
\[
  b_{3}^{-1}b_{2}^{-1}b_{1}^{-1}(a_1b_1a_{1}^{-1})(a_2b_2a_{2}^{-1})
  (a_3b_3a_{3}^{-1})=1, \]
  \[
  \beta_i=a=b=\phi(\beta_i)=\phi(a)=\phi(b)=1, \ \ (i=0,1,2,3).
\]
Since the monodromy of the Lefschetz fibration $M_{19}$ contains the factorization $W_3$,
the relations coming from the factorization $W_3$ make $\pi_{1}(M_{19})$ a quotient
of the free abelian group of rank 2. Hence, we can find additional relations coming from the
conjugated word $W_{3}^{\phi}$ using the action of $\phi$ on the generators of first homology
group $H_1(\Sigma_{3};\mathbb{Z})$. In addition to the relations $a_2=b_2=0$,
the relations (\ref{eq65})$-$(\ref{eq70}) give the following abelianized
relations in $H_{1}(M_{19};\mathbb{Z})$:
\begin{eqnarray}
 \label{eq71} a_1+a_3&=0,\\    
  \label{eq72}b_1+b_3&=0.
 \end{eqnarray}
The effect of the diffeomorphism $\phi$
on the generators $a_1$ and $a_3$ of $H_1(\Sigma_{3};\mathbb{Z})$ are as follows :
\begin{eqnarray}
\phi(a_1)=a_1-b_1-b_2-b_3,\label{phia1}\\
\phi(a_3)=a_3-b_1-b_2-2b_3, \label{phia2}
\end{eqnarray}
where the curves  $\phi(a_1)$ and $\phi(a_3)$ are shown in Figure~\ref{phicurves}.
From the fact that $\phi(\beta_1)=0$ and $\phi(\beta_0)=c_1=a_1=0$ in $H_1(M_{19};\mathbb{Z})$ and using the identities (\ref{eq65}) and (\ref{eq66}),
we get the relation
\[
\phi(\beta_1)=\phi(\beta_{0}+a_1+a_3)=\phi(a_1+a_3)=0.
\]
Hence, by the identities (\ref{phia1}) and (\ref{phia2}), the relation
\[
\phi(a_1)+\phi(a_3)=a_1+a_3-2b_1-2b_2-3b_3=0
\]
holds in $H_1(M_{19};\mathbb{Z})$, which implies that $b_3=0$ by the relations
(\ref{eq71}) and (\ref{eq72}). Also, using $\phi(\beta_0)=a_1=0$ and the relations
(\ref{eq71}), (\ref{eq72}), we conclude that $H_1(\Sigma_{3};\mathbb{Z})=0$.
This proves that $M_{19}$ is simply connected.
\end{proof}
\begin{figure}
\begin{center}
\scalebox{0.4}{\includegraphics{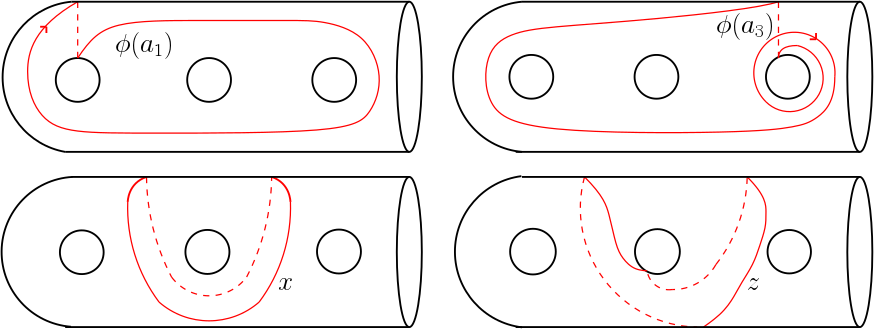}}
\caption{The curves $\phi(a_1)$, $\phi(a_3)$, $x$ and $z$ .}
\label{phicurves}
\end{center}
\end{figure}
\begin{lemma}\label{pim18}
The $4$-manifold  $M_{18}$ is simply-connected.
\end{lemma}

\begin{proof}
The monodromy of the Lefschetz fibration  $M_{18}$ is given as (\ref{m18}).
Hence $\pi_{1}(M_{18})$ has a presentation with generators $a_j,b_j$, $(j=1,2,3)$ and with the relations 
\begin{eqnarray*}
  b_{3}^{-1}b_{2}^{-1}b_{1}^{-1}(a_1b_1a_{1}^{-1})
 (a_2b_2a_{2}^{-1})(a_3b_3a_{3}^{-1})=1, \\
  \beta_i=a=b=c_3=C=B_2=\phi(\beta_k)=\phi(a)=\phi(b)=1,
\end{eqnarray*}
for each $i=0,1,2,3$ and $k=1,2,3$.
The relations $c_3=a_1a_{2}^{-1}=1$ and $a=a_2=1$ imply that $a_1=a_2=1$ in $\pi_{1}(M_{18})$. Hence, $\pi_{1}(M_{18})$ has the same
presentation as $\pi_{1}(M_{19})$. By the proof of 
Lemma~\ref{pim19}, $\pi_{1}(M_{18})$ is trivial.
\end{proof}
 \begin{theorem}\label{exotic1819}
 The $4$-manifolds $M_{18}$ and $M_{19}$ are exotic copies of the manifolds
 $3\C\P^{2} \#  18\overline{\C\P^{2}}$ and
 $3\C\P^{2} \#  19\overline{\C\P^{2}}$, respectively.
 \end{theorem}
 \begin{proof}
  The  $4$-manifolds $M_{18}$ and $M_{19}$ have the Euler characteristics
 \begin{eqnarray*}
e(M_{19})&=&e(\mathbb{S}^{2})e(\Sigma_{3})+\# \rm{singular \hspace{0.12cm} fibers}=2(-4)+32=24,\\
e(M_{18})&=&e(\mathbb{S}^{2})e(\Sigma_{3})+\# \rm{singular \hspace{0.12cm} fibers}=2(-4)+31=23,
\end{eqnarray*}
and the signatures 
\begin{eqnarray*}
\sigma(M_{19})&=&\sigma(M_3)+\sigma(M_3)=-16,\\
\sigma(M_{18})&=&\sigma(M_{19})+1=-15,
\end{eqnarray*}
 using the Novikov additivity and the fact that lantern substitution increases
 the signature by $1$ (Lemma~\ref{l1}(ii)). 

By Lemma~\ref{pim19} and Lemma~\ref{pim18}, $M_{19}$ and $M_{18}$
are simply connected. Hence the identities
\begin{eqnarray*}
e(M_{19})&=&2-2b_1(M_{19})+b_{2}(M_{19})=2+b_{2}^{+}(M_{19})+b_{2}^{-}(M_{19})=24,\\
\sigma(M_{19})&=&b_{2}^{+}(M_{19})-b_{2}^{-}(M_{19})=-16
\end{eqnarray*}
 imply that $(b_{2}^{+}(M_{19}),b_{2}^{-}(M_{19}))=(3,19)$.
 Using Freedman's work, we conclude that $M_{19}$ and $M_{18}$
 are homeomorphic to $3\C\P^{2} \#  19\overline{\C\P^{2}}$
 and $3\C\P^{2} \#  18\overline{\C\P^{2}}$, respectively.
  By similar arguments as in the proof of Theorem~\ref{t3} $M_{18}$ and  $M_{19}$ are minimal i.e., it cannot contain
 a smoothly embedded $-1$ sphere.  However, the manifolds
 $3\C\P^{2} \#  19\overline{\C\P^{2}}$ and
 $3\C\P^{2} \#  18\overline{\C\P^{2}}$ contain smoothly
 embedded $-1$ spheres, the exceptional spheres. Hence $M_{19}$ and
 $M_{18}$ cannot be diffeomorphic to 
$3\C\P^{2} \#  19\overline{\C\P^{2}}$ and
$3\C\P^{2} \#  18\overline{\C\P^{2}}$, respectively.
 \end{proof} 
 We now present the minimal symplectic genus-$3$ Lefschetz fibrations with
$(b_{2}^{+},b_{2}^{-})=(3,17)$ and $(3,16)$. To obtain such Lefschetz
fibrations, consider the following identity of the factorization $W$ in (\ref{eq3}):
\[
W= T_3t_{C}t_{C^{\prime}}=t_{\delta},
\] 
 where $T_3= t_{t_{C}^{-1}(B_{0})}t_{t_{C}^{-1}(B_{1})}t_{t_{C}^{-1}
 (B_{2})}t_{B_0}t_{B_1}t_{B_2}t_{B_{0}^{\prime}}t_{B_{1}^{\prime}}
 t_{B_{2}^{\prime}}t_{t_{C'}(B_{0}^{\prime})}t_{t_{C'}(B_{1}^{\prime})}
 t_{t_{C}^{\prime}(B_{2}^{\prime})}$.

Also, consider factorization $W_3$ 
 \[
 W_3=T_4t_{a}^{2}=t_{\delta},
 \]
 where $T_4=t_{\beta_{0}}t_{\beta_{1}}t_{\beta_{2}}t_{\beta_{3}}
 t_{a}^{2}t_{b}^{2}t_{\beta_{0}}t_{\beta_{1}}t_{\beta_{2}}t_{\beta_{3}} t_{b}^{2}$.
 
Thus we have 
 \begin{eqnarray}
W_3W =T_4t_{a}^{2}t_Ct_{C'} T_3=t_{\delta}^{2}.\label{mm17}
\end{eqnarray}
Since the curves $\lbrace a,a,C, C^{\prime}\rbrace$ bound a sphere with
four holes, we have the identity $t_{a}^{2}t_Ct_{C^{\prime}}=t_xt_bt_z$.
Therefore, we get
\begin{eqnarray}
T_4t_xt_bt_zT_3=t_{\delta}^{2}, \label{mm16}
\end{eqnarray}
where the Dehn twist curves $x$ and $z$ are depicted in Figure~\ref{phicurves}.
  Let $M_{17}$ and $M_{16}$  be the genus-$3$ Lefschetz fibrations with the
  monodromies (\ref{mm17}) and (\ref{mm16}), respectively.
 \begin{lemma}\label{pim17}
The $4$-manifold  $M_{17}$ is simply-connected.
 \end{lemma} 
  \begin{proof}
  The monodromy of the Lefschetz fibration  $M_{17}$ is given in (\ref{mm17}).
Hence, $\pi_{1}(M_{17})$ has a presentation with the generators $a_j,b_j$, $(j=1,2,3)$ and with the relations 
\begin{eqnarray*}
  b_{3}^{-1}b_{2}^{-1}b_{1}^{-1}(a_1b_1a_{1}^{-1})(a_2b_2a_{2}^{-1})
 (a_3b_3a_{3}^{-1})=1, \\
  \beta_i=a=b=B_k=B_{k}^{\prime}=C=C^{\prime}=1,
\end{eqnarray*}
 for each $i=0,1,2,3$ and $k=0,1,2$.
Therefore the relations (\ref{eq65})$-$(\ref{eq70}) and (\ref{eq4})$-$(\ref{eq11})
hold in $\pi_{1}(M_{17})$. 
These relations immediately imply that $\pi_1(M_{17})=1$.
\end{proof}
\begin{lemma}\label{pim16}
 The $4$-manifold  $M_{16}$ is simply-connected.
 \end{lemma} 
  \begin{proof}
  The monodromy of the Lefschetz fibration  $(M_{16},f_{16})$ is given in (\ref{mm16}).
  Hence, $\pi_{1}(M_{16})$ has a presentation with the generators $a_j$ and
$b_j$, $(j=1,2,3)$ and with the relations 
\begin{eqnarray*}
  b_{3}^{-1}b_{2}^{-1}b_{1}^{-1}(a_1b_1a_{1}^{-1})(a_2b_2a_{2}^{-1})
  (a_3b_3a_{3}^{-1})=1, \\
  \beta_i=a=b=B_k=B_{k}^{\prime}=x=z=1,
\end{eqnarray*}
for each $i=0,1,2,3$ and $k=0,1,2$.
Thus, the relations (\ref{eq65})-(\ref{eq70}) and (\ref{eq4})-(\ref{eq9}) hold
in $\pi_{1}(M_{16})$, which give rise to $\pi_1(M_{16})=1$.
\end{proof}
 \begin{theorem}
 The $4$-manifolds $M_{17}$ and $M_{16}$ are exotic copies of the manifolds
 $3\C\P^{2} \#  17\overline{\C\P^{2}}$ and
 $3\C\P^{2} \#  16\overline{\C\P^{2}}$, respectively.
 \end{theorem}
 \begin{proof}
  The manifolds $M_{17}$ and $M_{16}$ have the Euler characteristics
 \begin{eqnarray*}
e(M_{17})&=&e(\mathbb{S}^{2})e(\Sigma_{3})+\# \rm{singular \hspace{0.12cm} fibers}=2(-4)+30=22,\\
e(M_{16})&=&e(\mathbb{S}^{2})e(\Sigma_{3})+\# \rm{singular \hspace{0.12cm} fibers}=2(-4)+29=21
\end{eqnarray*}
and the signatures  
\begin{eqnarray*}
\sigma(M_{17})&=&\sigma(M_3)+\sigma(X)=-14,\\
\sigma(M_{16})&=&\sigma(M_{17})+1=-13,
\end{eqnarray*}
 using the Novikov additivity and  Lemma~\ref{l1}(ii).

It follows from Lemma~\ref{pim17} and Lemma~\ref{pim16}
that $M_{17}$ and $M_{16}$ are simply connected. Thus, the identities
\begin{eqnarray*}
e(M_{17})&=&2-2b_1(M_{17})+b_{2}(M_{17})=2+b_{2}^{+}(M_{17})+b_{2}^{-}(M_{17})
=22\\
\sigma(M_{17})&=&b_{2}^{+}(M_{17})-b_{2}^{-}(M_{17})=-14
\end{eqnarray*}
 imply that $(b_{2}^{+}(M_{17}),b_{2}^{-}(M_{17}))=(3,17)$.
Similarly, $(b_{2}^{+}(M_{16}),b_{2}^{-}(M_{16}))=(3,16)$.
Using Freedman's classification , we see that $M_{17}$ and $M_{16}$ are
homeomorphic to $3\C\P^{2} \#  17\overline{\C\P^{2}}$
and $3\C\P^{2} \#  16\overline{\C\P^{2}}$, respectively.
It is shown that the manifolds $M_{17}$ and $M_{16}$ are minimal in a similar way
in the proof of Theorem~\ref{exotic1819}. Therefore, $M_{17}$ and  $M_{16}$
cannot be diffeomorphic to $3\C\P^{2} \#  17\overline{\C\P^{2}}$
and $3\C\P^{2} \#  16\overline{\C\P^{2}}$, respectively.
  \end{proof} 
Now, let us consider the minimal genus-$3$ Lefschetz fibrations $(X_{i},f_i)$ $(i=1,2,3)$
constructed in Section~\ref{S3}.
\begin{theorem}
The $4$-manifolds $X_1$, $X_2$ and $X_3$ are exotic copies of
$3\C\P^{2} \#  15\overline{\C\P^{2}}$,
$3\C\P^{2} \#  14\overline{\C\P^{2}}$ and $3\C\P^{2} \#  13\overline{\C\P^{2}}$, respectively.
\end{theorem}
\begin{proof}
For each $i=1,2,3$, the manifold $X_i$ is minimal, simply connected
and has the following invariants:
\begin{eqnarray*}
e(X_i)&=&2-2b_1(X_i)+b_2(X_i)=2+b_{2}^{+}(X_i)+b_{2}^{-}(X_i)=21-i\\
\sigma(X_i)&=&b_{2}^{+}(X_i)+b_{2}^{-}(X_i)=-13+i.
\end{eqnarray*}
Therefore, $(b_{2}^{+}(X_i),b_{2}^{-}(X_i))=(3,16-i)$. It follows from Freedman's
classification that $X_{i}$ is homeomorphic to
$3\C\P^{2} \#  (16-i)\overline{\C\P^{2}}$.
Since each $X_i$ is minimal, $X_i$ cannot be diffeomorphic to
$3\mathbb{CP}^{2} \#  (16-i)\overline{\C\P^{2}}$. This finishes the proof.
 \end{proof}

%%%%%%%%%%%%%%%%%%%%%%%%%%%%%%%%%%%%%%%%%%%%%%%%%%%%%%%%%%%%%%%%%%%%%%
\section{Construction of genus-$3k$ Lefschetz fibrations and further examples of exotic $4$-manifolds}\label{S5}
In this section, we explicitly construct a positive factorization for a genus-$3k$ Lefschetz fibration over a sphere whose total space is diffeomorphic to
$\Sigma_{k}\times \mathbb{S}^{2}\# 6\overline{\C\P^{2}}$. In this construction, we use generalized Matsumoto's genus-$2k$ fibration and the similar idea as in Section~\ref{S3}. Moreover, we give some fibered and non-fibered examples of exotic structures using our generalized
construction via twisted fiber sum or Luttinger surgery.

\subsection{Construction genus-$3k$ Lefschetz fibrations from generalized Matsumoto's genus-$2k$ fibrations}\label{Xk}
We have the following two identities in $\Mod_{2k}^{2}$ using the liftings of
generalized Matsumoto's fibration for even $g$  given by Hamada~\cite{h}:
 \begin{eqnarray}
\label{eq5.1}(t_{B_{0}}t_{B_{1}}t_{B_{2}}\ldots t_{B_{2k}}t_{C})^{2}&=&(t_{C}t_{B_{0}}
t_{B_{1}}t_{B_{2}}\cdots t_{B_{2k}})^{2}=t_{C^{\prime}}\\
\label{eq5.2}(t_{B_{0}^{\prime}}t_{B_{1}^{\prime}}t_{B_{2}^{\prime}}\cdots t_{B_{2k}^{\prime}}
t_{C^{\prime}})^{2}&=&t_{C}t_{\delta},
 \end{eqnarray}
 Here the curves $B_i$, $B_{i}^{\prime},C,C^{\prime}$ are as shown
in Figure~\ref{GMAT} and $\delta$ is the curve that is parallel to the
boundary component of $\Sigma_{3k}^{1}$. The first identity in (\ref{eq5.1}) holds by the
commutativity of the Dehn twists about disjoint curves $C$ and $C^{\prime}$. The second identity in (\ref{eq5.1}) can be obtained by capping off the boundary component $\delta_1$ in Figure~\ref{GM}. We embed these curves into
$\Sigma_{3k}^{1}$ as in Figure~\ref{GMAT}, and use the commutativity of $t_{C}$ and $t_{C^{\prime}}$ to obtain the following relation in $\Mod_{3k}^{1}$.
\begin{eqnarray*}
t_{\delta}&=&t_{B_{0}}t_{B_{1}}\cdots t_{B_{2k}}t_{C}t_{B_{0}}t_{B_{1}}\cdots t_{B_{2k}}
 t_{B_{0}^{\prime}}t_{B_{1}^{\prime}}\cdots t_{B_{2k}^{^{\prime}}}
 t_{C^{\prime}}t_{B_{0}^{\prime}}t_{B_{1}^{\prime}}\cdots t_{B_{2k}^{^{\prime}}}
 t_{C^{\prime}}t_{C}t_{C}^{-1}t_{C\prime}^{-1}\\
&=&t_{B_{0}}t_{B_{1}}\cdots t_{B_{2k}}t_{C}t_{B_{0}}t_{B_{1}}\cdots t_{B_{2k}}
t_{B_{0}^{\prime}}t_{B_{1}^{\prime}}\cdots t_{B_{2k}^{^{\prime}}}t_{C^{\prime}}
t_{B_{0}^{\prime}}t_{B_{1}^{\prime}}\cdots t_{B_{2k}^{^{\prime}}}.
 \end{eqnarray*}
\begin{figure}
\begin{center}
\scalebox{0.4}{\includegraphics{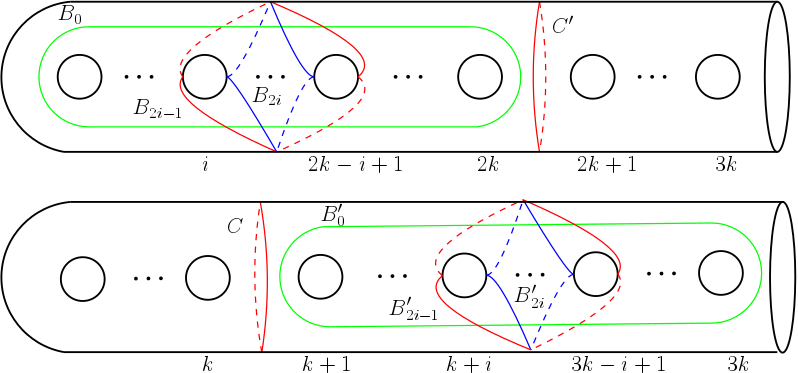}}
\caption{The curves for the monodromy $W_k$}
\label{GMAT}
\end{center}
\end{figure}
Let us denote the above factorization by $W_k$. 
Note that $W_k$ is a product consisting of positive Dehn twists about $8k+4$ nonseparating and two separating curves. Let $X(k)$
be the total space of the genus-$3k$ Lefschetz fibration over $\mathbb{S}^2$ with the monodromy factorization $W_k$. By applying
the technique of Endo and Nagami~\cite{e11} to compute the signature $\sigma(X(k))$
of $X(k)$ and the Euler characteristic formula for the Lefschetz fibrations, we get the following topological invariants of $X(k)$.
\begin{eqnarray*}
 e(X(k))&=&e(\mathbb{S}^{2})e(\Sigma_{3k})+\# \rm{singular \hspace{0.12cm} fibers}=2(2-6k)+8k+6=-4k+10,\\
 \sigma(X(k))&=&I_{3k}\big((B_{0}B_{1}\cdots B_{2k} C)^{2}(B_{0}^{\prime}
B_{1}^{\prime}\cdots B_{2k}^{\prime}C^{\prime})^{2}C^{-1}
 C^{{\prime}^{-1}}\big)  \\
  &=&I_{3k}\big((B_{0}B_{1}\cdots B_{2k} C)^{2}\big)+I_{3k}\big((B_{0}^{\prime}
 B_{1}^{\prime}\cdots B_{2k}^{\prime}C^{\prime})^{2}\big)-I_{3k}(C)-
  I_{3k}(C^{\prime})\\
  &=&-4-4-(-1)-(-1)=-6,\\
 c_{1}^{2}\big(X(k)\big)&=&3\sigma(X(k))+2e\big(X(k)\big)=-8k+2.
 \end{eqnarray*}

\begin{figure}
\begin{center}
\scalebox{0.45}{\includegraphics{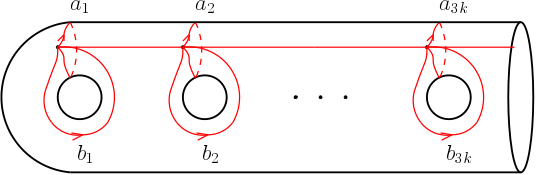}}
\caption{The generators of  $\pi_{1}(\Sigma_{3k})$}
\label{GMATFUND}
\end{center}
\end{figure}

\begin{lemma} \label{lem} 
For each $1\leq i \leq k$ and $1\leq j \leq k-1$, $\pi_{1}(X(k))$ has the following relations:
\begin{gather} 
 \label{eq76} a_{i}a_{2k-i+1}=1,\hspace{0.2cm} a_{k+i}a_{3k-i+1}=1,\\
 \label{eq77} b_{j+1}b_{j+2} \cdots b_{2k-j}=[a_{2k+1-j},b_{2k+1-j}]
[a_{2k+2-j},b_{2k+2-j}] \cdots [a_{2k},b_{2k}],\\
 \label{eq78} b_{k+j+1}b_{k+j+2} \cdots b_{3k-j}=[a_{3k+1-j},b_{3k+1-j}]
[a_{3k+2-j},b_{3k+2-j}] \cdots [a_{3k},b_{3k}].
\end{gather}
\begin{proof}
Let $a_i$ and $b_i$ be the generators of $\pi_{1}(\Sigma_{3k})$ for
$i=1, \ldots,3k$ as shown in Figure~\ref{GMATFUND}. Since the genus-$k$
Lefschetz fibration $X(k)\to \mathbb{S}^{2}$ admits a section, $\pi_{1}(X(k))$
is isomorphic to the quotient of $\pi_{1}(\Sigma_{3k})$ by the normal closure
of the vanishing cycles. Therefore, the fundamental group $\pi_{1}(X(k))$ has the following relations up to
conjugation:
\begin{eqnarray*}
B_0&=&b_1b_2\cdots b_{2k}=1,\\
B_{2i-1}&=&a_{i}b_{i}b_{i+1}\cdots b_{2k+1-i}c_{2k+1-i}
a_{2k+1-i}=1,\hspace{0.2cm} 1\leq i\leq k,\\
B_{2i}&=&a_{i}b_{i+1}b_{i+2}\cdots b_{2k-i}c_{2k-i}
a_{2k+1-i}=1,\hspace{0.2cm} 1\leq i\leq k-1,\\
B_{2k}&=&a_{k}c_{k}a_{k+1}=1,\\
B_{0}^{\prime}&=&b_{k+1}b_{k+2}\cdots b_{3k}=1,\\
B_{2i-1}^{\prime}&=&a_{k+i}b_{k+i}b_{k+i+1}\cdots b_{3k+1-i}
c_{3k+1-i}a_{3k+1-i}=1,\hspace{0.2cm} 1\leq i\leq k,\\
B_{2i}^{\prime}&=&a_{k+i}b_{k+i+1}b_{k+i+2}\cdots b_{3k-i}
c_{3k-i}a_{3k+1-i}=1,\hspace{0.2cm} 1\leq i\leq k-1,\\
B_{2k}^{\prime}&=&a_{2k}c_{2k}a_{2k+1}=1,\\
C&=&c_k=1,\\
D&=&c_{2k}=1,\\
c_{3k}&=&[a_{1},b_{1}][a_{2},b_{2}]\cdots[a_{3k},b_{3k}]=1,
\end{eqnarray*}
where $c_{j}=[a_{1},b_{1}][a_{2},b_{2}]\cdots[a_{j},b_{j}]$ for $1\leq j \leq 3k$.\par
  First consider the relations $B_{1}=a_{1}b_{1}b_{2}\cdots b_{2k}c_{2k}a_{2k}=1$, $B_0=b_1b_2\cdots b_{2k}=1$
  and $D= c_{2k}=1$. Then one can easily get $a_1a_{2k}=1$. Now consider 
$B_{2}=a_{1}b_{2}b_{3}\cdots b_{2k-1}c_{2k-1}a_{2k}=1$. Using the equation $a_1a_{2k}=1$, we get $b_{2}b_{3}\cdots b_{2k-1}c_{2k-1}=1$. This implies that $a_2a_{2k-1}=1$ using the relation $B_{3}=a_{2}b_{2}\cdots b_{2k-1}c_{2k-1}a_{2k-1}=1$.
 Inductively, one can obtain the relations $a_ia_{2k-i+1}=1$ for $1\leq i \leq k$.\par
   To get the relations $a_{k+i}a_{3k-i+1}=1$ for $1\leq i \leq k$, combine the equations 
$
  B_{1}^{\prime}=a_{k+1}b_{k+1}b_{k+2}\cdots b_{3k}c_{3k}a_{3k}=1$,
   $B_{0}^{\prime}=b_{k+1}b_{k+2}\cdots b_{3k}=1$ and $c_{3k}=1$, which implies that $a_{k+1}a_{3k}=1$. Then, using the relation $B_{2}^{\prime}=a_{k+1}b_{k+2}b_{k+3}\cdots b_{3k-1}c_{3k-1}a_{3k}=1$ together with the equation $a_{k+1}a_{3k}=1$, we obtain $b_{k+2}b_{k+3}\cdots b_{3k-1}c_{3k-1}=1$.
  Then, by inserting it into $B_{3}^{\prime}=a_{k+2}b_{k+2}b_{k+2}\cdots b_{3k}c_{3k-1}a_{3k-1}=1$, we have $a_{k+2}a_{3k-1}=1$. Continuing in this way, we conclude the desired
relation $a_{k+i}a_{3k-i+1}=1$ for $1\leq i \leq k$.\par
We next show that  $b_{i+1}b_{i+2} \cdots b_{2k-i}=[a_{2k+1-i},b_{2k+1-i}]
[a_{2k+2-i},b_{2k+2-i}] \cdots [a_{2k},b_{2k}]$ for $1\leq i \leq k-1$.
It follows from the equations $B_{2i}=a_{i}b_{i+1}b_{i+2}\cdots b_{2k-i}c_{2k-i}a_{2k+1-i}=1$ and $a_{i}a_{2k-i+1}=1$ that $b_{i+1}b_{i+2} \cdots b_{2k-i}=c_{2k-i}^{-1}$,
for $1\leq i \leq k-1$. This actually gives the required relation by the definition of $c_{2k-i}$.\par
The last equation (\ref{eq78}) comes from the relations $a_{k+i}a_{3k-i+1}=1$ and $B_{2i}^{\prime}=a_{k+i}b_{k+i+1}b_{k+i+2}\cdots b_{3k-i}c_{3k-i}a_{3k+1-i}=1$ for $1\leq i \leq k-1$ in a similar way. 
\end{proof}
\end{lemma}
\begin{corollary}
The first homology group $H_1(X(k);\mathbb{Z})$ of $X(k)$ is isomorphic to $\mathbb{Z}^{2k}$.
\begin{proof}
Observe that $a_{k+i}$ for all $i=1,\ldots,2k$ can be written in terms of
$a_j$ for $j=1,\ldots,k$  by the relations in (\ref{eq76}). The abelianization of the relation (\ref{eq77}) gives rise to the equation
$b_{j+1}=-b_{2k-j}$ for each $j=1,\ldots,k-1$. Furthermore, it follows from the
abelianization of the relation (\ref{eq78})  that $b_{k+j+1}=-b_{3k-j}$ for each
$j=1,\ldots,k-1$. One can easily observe that $b_{k+i}$ for all $i=1,\ldots,2k$
can be written also in terms of $b_j$ for $j=1,\ldots,k$. This finishes the proof.
\end{proof}
\end{corollary}
\begin{theorem}\label{txk}
 The $4$-manifold $X(k)$ is diffeomorphic to
 $\Sigma_{k}\times \mathbb{S}^{2}\# 6\overline{\C\P^{2}}$ for any positive integer $k$.
\end{theorem}
\begin{proof}
Using $H_1(X(k);\mathbb{Z})\cong \mathbb{Z}^{2k}$ and some other topological invariants that we obtain above, one can easily compute that
$b_1(X(k))=2k$ and $b_2^{+}(X(k))=1$. When $k=1$, we showed that
the total space $X$ of the Lefschetz fibration with the factorization $W=W_1$
is diffeomorphic to $\mathbb{T}^{2}\times \mathbb{S}^{2}\# 6\overline{\C\P^{2}}$
in Theorem~\ref{t6}.\par
When $k=2$, we prove that $X(2)$ is diffeomorphic to a blow-up of a ruled surface. Assuming that $X(2)$ is not diffeomorphic to (a blow-up of) a ruled surface,
then $X(2)\cong \widetilde{X(2)}\# m \overline{\C\P^{2}}$ where
$\widetilde{X(2)}$ is the minimal model of $X(2)$ and $m$ is some non-negative integer.
It is easily computed that 
\begin{eqnarray*}
c_{1}^{2}(\widetilde{X(2)})=c_{1}^{2}(X(2))+m=-14+m.
\end{eqnarray*}
Since the minimal $4$-manifold $X(2)$ is neither rational nor ruled, we have
 \[
 c_{1}^{2}(\widetilde {X(2)})=-14+m \geq 0,
 \]
which implies that $m\geq 14$ ~\cite{t2,liu2}.
On the other hand, it follows from~\cite[Theorem 4.7]{sa1} that $m \leq 2g-2=10$,
where $g$ is the genus of the Lefschetz fibration $X(2) \to \mathbb{S}^{2}$. (Thus $g=6$.)
Since $X(2)$ is neither rational nor ruled, it admits $m$ disjoint exceptional spheres. This yields a contradiction.\par
When $k>2$, it follows from $e(X(k))<0$ that $X(k)$ is diffeomorphic to a blow-up of a
ruled surface~\cite{St.1}. Thus, using the signature and Euler characteristic of $X(k)$,
we can deduce that $X(k)$ is diffeomorphic to 
$\Sigma_{k}\times \mathbb{S}^{2}\# 6\overline{\C\P^{2}}$.
Therefore, $\pi_{1}(X(k))$, having the representation in Lemma ~\ref{lem} is isomophic to the surface group $\pi_{1}(\Sigma_{k})$. 
\end{proof}

\subsubsection{\textbf{Construction of fibered exotic $(4k^{2}-2k+1)\C\P^{2} \#  (4k^{2}+4k+7)\overline{\C\P^{2}}$, using genus-$3k$ fibration on $\Sigma_{k}\times \mathbb{S}^{2}\# 6\overline{\C\P^{2}}$}}
In order to produce exotic $4$-manifolds 
$(4k^{2}-2k+1) \C\P^{2} \#  (4k^{2}+4k+7)\overline{\C\P^{2}}$ 
which carry the genus-$3k$ Lefschetz fibration structure, we perform sufficiently many twisted
fiber sums of the genus-$3k$ Lefschetz fibration $X(k)$ to get a  simply-connected
$4$-manifold.\par

\begin{theorem}
There exist minimal symplectic exotic copies of
$(4k^{2}-2k+1) \C\P^{2} \#  (4k^{2}+4k+7)\overline{\C\P^{2}}$
admitting genus-$3k$ Lefschetz fibration structure for each integer $k\geq1$.
\end{theorem}
\begin{proof}
We start with the Lefschetz fibration
$X(k)=\Sigma_{k}\times \mathbb{S}^{2}\# 6\overline{\C\P^{2}}$
with monodromy $W_k$. We can choose a diffeomorphism in such a way that
when we perform twisted fiber sum of $W_k$, the word induced by conjugating
$W_k$ with this diffeomorphism kills some generators of the fundamental group
$\pi_1(X(k))$. Considering the disjoint vanishing cycles $B_{2k}$ and $B_{2k}^{\prime}$,
one can find diffeomorphisms $f_i$ such that for each $i=1,\ldots,k-1$
\[
f_{i}(B_{2k})=a_i, f_{i}(B_{2k}^{\prime})=b_{i+1},f_{k}(B_{2k})=a_k \textrm{ and }f_{i}(B_{2k}^{\prime})=b_1,
\]
by the classification of surfaces where $a_i$, $b_i$'s are the generators of $\pi_1(X(k))$
as shown in Figure~\ref{GMATFUND} and the curves $B_{2k}$ and $B_{2k}^{\prime}$ are shown in Figure~\ref{GMAT}. We obtain the monodromy factorization
\[
W_{k}W_{k}^{f_1}\ldots W_{k}^{f_k}
\]
by conjugating the diffeomorphisms $f_i$.
Let $X(k,k+1)$ be the Lefschetz fibration with
the monodromy factorication $W_{k}W_{k}^{f_1}\cdots W_{k}^{f_k}$. Using the
theory of Lefschetz fibrations and the existence of a section, $\pi_1(X(k,k+1))$ is a quotient of $\pi_1(X(k))\cong\pi_1(\Sigma_k)$. The conjugated words
$W_{i}^{f_i}$ induce the additional relations containing 
\[
f_{i}(B_{2k})=a_i=1, f_{i}(B_{2k}^{\prime})=b_{i+1}=1 \textrm{ for each } i=1,\ldots, k-1, 
\]
\[
f_{k}(B_{2k})=a_k=1 \textrm{ and }
f_{i}(B_{2k}^{\prime})=b_1=1.
\]
 Hence, the additional relations induced by
conjugated words kill all generators $a_i,b_i$ for $i=1,\ldots,k$,
which implies that $\pi_1(X(k,k+1))$ is trivial. Using the fiber sum computations, one can also compute 
\begin{eqnarray*}
e(X(k,k+1))&=&e(\mathbb{S}^{2})e(\Sigma_{3k})+\# \rm{singular \hspace{0.12cm} fibers}=2(2-6k)+(k+1)(8k+6)=8k^{2}+2k+10,\\
\sigma(X(k,k+1))&=&(k+1)\sigma(X(k))=(k+1)(-6)=-6k-6.
\end{eqnarray*}
Freedman's classification implies that $X(k,k+1)$ is homeomorphic to
$(4k^{2}-2k+1) \C\P^{2} \#  (4k^{2}+4k+7)\overline{\C\P^{2}}$
for any integer $k>0$. Theorem~\ref{t1} (or Proposition~\ref{p1}) implies that for each $k$, 
 $X(k,k+1)$ is minimal, and therefore they are not diffeomorphic to $(4k^{2}-2k+1) \C\P^{2} \#  (4k^{2}+4k+7)\overline{\C\P^{2}}$ for any integer $k>0$.
\end{proof}
\begin{remark}
Further fibered minimal exotic examples can be produced using other genus-$3k$ Lefschetz fibrations over $\mathbb{S}^{2}$ and performing lantern substitutions. 
\end{remark}

\subsubsection{\textbf{Construction of exotic, not fibered, $(4k-1) \C\P^{2} \#  (4k+5)\overline{\C\P^{2}}$ using genus-$3k$ fibration on $\Sigma_{k}\times \mathbb{S}^{2}\# 6\overline{\C\P^{2}}$}}

To construct exotic copies of $(4k-1) \C\P^{2} \#  (4k+5)\overline{\C\P^{2}}$
for any positive integer $k$, we use the following family of symplectic building block.
It is obtained from $\Sigma_{3k}\times \mathbb{T}^{2}$ by performing a sequence
of torus surgeries. Our computations are similar to the some computations in~\cite{as}. We perform the following $6k$-torus surgeries on
$\Sigma_{3k}\times \mathbb{T}^{2}$ for fixed integers $m,k\geq 1$ and $p,q\geq 0$:
\begin{align*}
&(\beta_{1}^{\prime}\times c^{\prime\prime},\beta_{2k},-1), 
(\alpha_{3k}^{\prime \prime}\times d^{\prime},d^{\prime}, m/q),\\
&(\beta_{2}^{\prime}\times c^{\prime\prime},\beta_{2k+1},-1), 
(\alpha_{1}^{\prime}\times c^{\prime},\alpha_{1}^{\prime},-1),\\
&(\beta_{3}^{\prime}\times c^{\prime\prime},\beta_{2k+2},-1), 
(\alpha_{2}^{\prime}\times c^{\prime},\alpha_{2}^{\prime},-1),\\
&\ldots , \ldots \\
&(\beta_{k}^{\prime}\times c^{\prime\prime},\beta_{3k-1},-1), 
(\alpha_{k-1}^{\prime}\times c^{\prime},\alpha_{k-1}^{\prime},-1), \\
&(\beta_{k+1}^{\prime}\times c^{\prime\prime},\beta_{1},-1), 
(\alpha_{k}^{\prime}\times c^{\prime},\alpha_{k}^{\prime},-1),\\
&(\beta_{k+2}^{\prime}\times c^{\prime\prime},\beta_{2},-1), 
(\alpha_{k+1}^{\prime}\times c^{\prime},\alpha_{k+1}^{\prime},-1),\\
&\ldots , \ldots \\
&(\beta_{2k}^{\prime}\times c^{\prime\prime},\beta_{k},-1),
 (\alpha_{2k-1}^{\prime}\times c^{\prime},\alpha_{2k-1}^{\prime},-1),\\
&(\beta_{2k+1}^{\prime}\times c^{\prime\prime},\beta_{k+1},-1),
 (\alpha_{2k}^{\prime}\times c^{\prime},\alpha_{2k}^{\prime},-1),\\
&(\beta_{2k+2}^{\prime}\times c^{\prime\prime},\beta_{k+2},-1), 
(\alpha_{2k+1}^{\prime}\times c^{\prime},\alpha_{2k+1}^{\prime},-1),\\
&\ldots , \ldots \\
&(\beta_{3k-1}^{\prime}\times c^{\prime\prime},\beta_{2k-1},-1), 
(\alpha_{3k-2}^{\prime}\times c^{\prime},\alpha_{3k-2}^{\prime},-1),\\
&(\alpha_{3k}^{\prime}\times c^{\prime},c^{\prime},1/p), 
(\alpha_{3k-1}^{\prime}\times c^{\prime},\alpha_{3k-1}^{\prime},-1),
\end{align*}
where $\alpha_i,\beta_i$ are the generators of
$\pi_{1}(\Sigma_{3k})$ for $i=1,2\ldots 3k$ and $c,d$ are the generators of $\pi_{1}(\mathbb{T}^{2})$. Let $Y_k(1/p , m/q)$ denote the resulting smooth $4$-manifold. When $m=1$, the above torus surgeries are Luttinger surgeries
and in this case the Luttinger surgery preserves the minimality. Moreover, it can be performed symplectically. \par
\begin{figure}[h]
\begin{center}
\scalebox{0.4}{\includegraphics{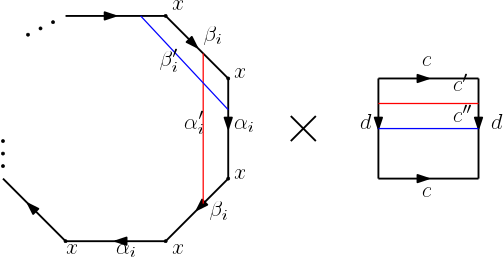}}
\caption{Lagrangian tori $\beta_{i}^{\prime}\times c^{\prime \prime}$ and
$\alpha_{i}^{\prime}\times c^{\prime}$ }
\label{lagrangian}
\end{center}
\end{figure}
The fundamental group of the manifold $Y_k(1/p , m/q)$ is
generated by $\alpha_i,\beta_i,c$ and $d$ for $i=1,2\ldots 3k$ and it has the following relations:
\[ [\alpha_{1}^{-1},d]=\beta_{2k}, [\alpha_{2}^{-1},d]=
\beta_{2k+1},\cdots ,[\alpha_{k}^{-1},d]=\beta_{3k-1},\]
\[ [\alpha_{k+1}^{-1},d]=\beta_{1},[\alpha_{k+2}^{-1},d]=
\beta_{2},\cdots ,[\alpha_{3k-1}^{-1},d]=\beta_{2k-1},\]
\[ [c^{-1},\beta_{3k}]^{-m}=d^{q}, [\beta_{1}^{-1},d^{-1}] =
a_1, [\beta_{2}^{-1},d^{-1}] =\alpha_2, \]

\[ \cdots, \cdots\]

\[ [\beta_{3k-1}^{-1},d^{-1}] =\alpha_{3k-1}, [d^{-1},\beta_{3k}^{-1}] 
=c^{p}, [\beta_j,c]=1, \]
\[ [\alpha_{3k},d]=1, [\alpha_j,c]=1, [\alpha_{3k},c]=1,\] 
\[ [\alpha_{1},\beta_{1}][\alpha_{1},\beta_{1}] \cdots [\alpha_{3k},\beta_{3k}]
=1, [c,d]=1. \addtag \label{yrel}  \]
where $1\leq j\leq 3k-1$.

\begin{theorem}
There exist smooth exotic copies of
 $(4k-1) \C\P^{2} \#  (4k+5)\overline{\C\P^{2}}$.
\end{theorem}
\begin{proof}
Consider the symplectic manifold $X(k)$ with a genus-$3k$ symplectic
submanifold $\Sigma_{3k}$, a regular fiber coming from its genus-$3k$
Lefschetz fibration structure and the symplectic $4$-manifold $Y_k(1,1)$
 with the symplectic submanifold $\Sigma_{3k}^{\prime}$ that is a copy of
$\Sigma_{3k}\times\lbrace pt \rbrace$ in $Y_k(1,1)$. (Here $Y_k(1,1)$ is obtained by performing torus surgeries from
$\Sigma_{3k}\times \mathbb{T}^{2}$ where $p=q=m=1$ as explained above). Let $Z(k)$ be the
$4$-manifold obtained by symplectic fiber sum of $X(k)$ and $Y_k(1,1)$
along the surfaces $\Sigma_{3k}$  and $\Sigma_{3k}^{\prime}$. We need
to find an orientation-reversing gluing diffeomorphism to perform symplectic
fiber sum such that  $Z(k)$ is simply-connected. \par
Recall from Lemma~\ref{lem} that $a_i,b_i$ ($i=1,\ldots,3k$) are the generators
of the fundamental group $\pi_{1}(X(k))\cong \pi_{1}(\Sigma_{k})$ but the generators
$a_i,b_i$ ($k+1\leq i \leq3k$) are nullhomotopic. The genus-${3k}$ Lefschetz fibration $X(k)\to \mathbb{S}^{2}$ admits a section, so the normal circle to
$\Sigma_{3k}$, denote it by $\lambda$,
 is nullhomotopic in $\pi_{1}(X(k)\setminus \upsilon\Sigma_{3k})$. We thus get $\pi_{1}(X(k)\setminus \upsilon\Sigma_{3k})$ is isomorphic to $\pi_{1}(X(k))$.\par
The fundamental group $\pi_{1}(Y_{k}(1,1))$ has the generators $\alpha_i, \beta_i,c$ and $d$ for $i=1,\ldots,3k$ with the relations in (\ref{yrel}). Choose a base point $p$ of
$\pi_{1}(Y_{k}(1,1))$  on 
$\partial \upsilon \Sigma_{3k}^{\prime}\cong\Sigma_{3k}^{\prime}\times S^{1}$
in such a way that $\pi_{1}(Y_{k}(1,1)\setminus\Sigma_{3k}^{\prime},p)$ is
normally generated by $\alpha_i,\beta_i,c$ and $d$ for $i=1,\ldots,3k$. 
One can perform above tori surgeries such that $\Sigma_{3k}^{\prime}\subset Y_{k}(1,1)$
is disjoint from all tori surgeries performed. Hence the relations in (\ref{yrel}) still hold in
$\pi_{1}(Y_{k}(1,1)\setminus\Sigma_{3k}^{\prime},p)$ except for $[c,d]=1$, 
which represents a meridian in $\pi_{1}(Y_{k}(1,1)\setminus\Sigma_{3k}^{\prime},p)$, denote it by 
$\lambda^{\prime}$.\par
 Now, choose the gluing diffeomorphism 
 $\varphi: \partial (\Sigma_{3k}) \to \partial (\Sigma_{3k}^{\prime})$ which maps the generators  of $\pi_{1}$ as follows:
 \[ 
 a_i\mapsto \alpha_{i}, b_i\mapsto \beta_{i}, \textrm{ and }\lambda\mapsto \lambda^{\prime}.
 \]
By Van Kampen's theorem, the fundamental group of the resulting $4$-manifold
 $Z(k)=(X(k)\setminus\upsilon\Sigma_{3k}\cup_{\varphi}Y_{k}(1,1)\setminus\upsilon\Sigma_{3k}^{\prime})$
satisfies
 \[
 \pi_{1}(Z(k))\cong\dfrac{\pi_{1}(X(k)\setminus \upsilon
 \Sigma_{3k}\ast \pi_{1}(Y_{k}(1,1)\setminus\Sigma_{3k}^{\prime})}
 {\langle a_i=\alpha_i, b_i=\beta_i, \lambda=\lambda^{\prime}\rangle} .
 \]

One can conclude that $\pi_{1}(Z(k))$ admits a presentation with generators
$a_i, b_i,c$ and $d$ ($i=1,\ldots,3k$)  and with the relations
(\ref{eq76}), (\ref{eq77}), (\ref{eq78}) and  (\ref{yrel}). Keep in mind $a_i=\alpha_i$, $b_i=\beta_i$, $\lambda=\lambda^{\prime}$
and $[c,d]=\lambda^{\prime}$. It suffices to prove that $c=d=1$ in 
$\pi_{1}(Z(k))$ to show that $\pi_{1}(X(k))$ is trivial using the relations
in (\ref{yrel}). Now, first consider the relations $[a_{3k},d]=1, [a_{k+1}^{-1},d]=b_1$ in (\ref{yrel})
 and the relation $ a_{3k}=a_{k+1}^{-1}$ in (\ref{eq76}), which yield $b_1=1$. Next, since $[b_{1}^{-1},d^{-1}]=a_1$ in (\ref{yrel}), $a_1=1$. 
  Also, using the relations $a_1a_{2k}=a_{2k}a_{2k+1}=1$ in (\ref{eq76}), we have $a_{2k}=a_{2k+1}=1$. Using these equations,
one can obtain that $b_k=b_{k+1}=1$, since $[a_{2k}^{-1},d]=b_k$ and $[a_{2k+1}^{-1},d]=b_{k+1}$ in (\ref{yrel}). It follows from the relations $
 [b_{k}^{-1},d^{-1}]=a_k$  and $b_k=1$ in (\ref{yrel}) that $a_k=1$. Also, by the fact that $[a_{k}^{-1},d]=b_{3k-1}$ in (\ref{yrel}) and $a_k=1$, we obtain that $b_{3k-1}=1$. This also implies that $a_{3k-1}=1$ by the relation $[b_{3k-1}^{-1},d^{-1}]=a_{3k-1}$ in (\ref{yrel}). The equations $a_k=1$ and $[a_{k},a_{k+1}]=[a_{k+1},a_{3k}]=1$ in (\ref{eq76}) imply that $  a_{3k}=1$. Now, consider the relation $b_{k+2}b_{k+3}\cdots b_{3k-1}=[a_{3k-1},b_{3k-1}][a_{3k},b_{3k}$ in (\ref{eq78}) and the relation $b_{k+1}b_{k+2}\cdots b_{3k}=1$ coming from vanishing cycle $B_{0}^{\prime}$ shown 
    in Figure~\ref{GMAT}. Combining these relations, we get $b_{k+2}b_{k+3}\cdots b_{3k-1}=b_{k+1}b_{k+2}\cdots b_{3k}=1$, since $a_{3k}=1$ and $b_{3k-1}=1$. Then, we get $b_{k+1}b_{3k}=1$.
       This impiles that $b_{3k}=1$ since $b_{k+1}=1$.
       Finally, we can obtain that $c=d=1$ using the relations $[c^{-1},b_{3k}]=d$ and $[d^{-1},b_{3k}^{-1}]=c$ in (\ref{yrel}). Therefore, the following relations
\[ [b_{i}^{-1},d^{-1}]=a_i, \hspace{0.4cm} i=1,\ldots, 3k-1\]
\[ [a_{i}^{-1},d]=b_{2k-1+i}, \hspace{0.4cm} i=1, \ldots, k\]
\[ [a_{i}^{-1},d]=b_{i-k}, \hspace{0.4cm} i=k+1, \ldots, 3k-1\]
in (\ref{yrel}) imply that $\pi_{1}(Z(k))=1$ . \par

Using the fact that Luttinger surgery preserves the Euler characteristic
and the signature, we have
\begin{eqnarray*}
e(Z(k))&=&e(X(k))+e(Y_k(1,1))-2e(\Sigma_{3k})=6+8k,\\
\sigma(Z(k))&=&\sigma(Z(k))+\sigma(Y_k(1,1))=-6.
\end{eqnarray*}
The Freedman's classification implies that the $4$-manifold $Z(k)$ is homeomorphic to 
$(4k-1) \C\P^{2} \#  (4k+5)\overline{\C\P^{2}}$ for
any positive integer $k$. Since the symplectic $4$-manifold $Z(k)$ has $b_{2}^{+}(Z(k))\geq 2$, the Seiberg-Witten invariant of the canonical class of $Z(k)$  is $\pm 1$~\cite{t3}. 
However, the Seiberg-Witten invariant of the canonical class of 
$(4k-1) \C\P^{2} \#  (4k+5)\overline{\C\P^{2}}$ is trivial
(cf.~\cite{sala}). Hence we distinguish $Z(k)$ with 
$(4k-1) \C\P^{2} \#  (4k+5)\overline{\C\P^{2}}$
up to diffeomorphism, which is due to the fact that Seiberg-Witten invariant is a diffeomorphism invariant. Moreover, one can replace $4$-manifold $Y_k(1,1)$ by $Y_k(1,m)$ in 
the construction above, where the integer $m\neq1$ to construct infinitely
many exotic copies of $(4k-1) \C\P^{2} \#  (4k+5)\overline{\C\P^{2}}$. 
\end{proof}

%%%%%%%%%%%%%%%%%%%%%%%%%%%%%%%%%%%%%%%%%%%%%%%%%%%%%%%%%%%%%%%%%%%%

\end{document}